\documentclass{article}
\usepackage{arxiv}

\usepackage{amsmath, amssymb}
\usepackage{graphicx}
\usepackage{xcolor}
\usepackage{amsthm}
\usepackage{stmaryrd}
\usepackage{siunitx}
\usepackage{caption}
\usepackage{subcaption}

\newtheorem{theorem}{Theorem}
\newtheorem{lemma}[theorem]{Lemma}

\theoremstyle{definition}
\newtheorem{definition}[theorem]{Definition}
\newtheorem{remark}[theorem]{Remark}

\DeclareMathOperator{\laplace}{\Delta}

\newcommand{\R}{\mathbb{R}}

\newcommand{\N}{\mathbb{N}}
\newcommand{\w}{\kappa}
\newcommand{\x}{\mathbf{x}}

\newcommand{\domain}{\Omega}
\newcommand{\bnd}{{\partial \domain}}
\newcommand{\elbnd}{{\partial T}}
\newcommand{\ltwo}{L^{2}}

\DeclareMathOperator{\diver}{div}
\newcommand{\hdiv}{H (\diver, \mathcal{T})}
\newcommand{\hdivel}{H(\diver, T)}
\newcommand{\conjugate}[1]{\overline{#1}}
\newcommand{\normal}{\mathbf{n}}
\newcommand{\n}{\normal}
\DeclareMathOperator{\grad}{\nabla}

\title{Error Analysis of an HDG Method with Impedance Traces for the Helmholtz Equation}
\date{\today}
\author{Michael Leumüller\\
	Institute of Analysis and Scientific Computing\\
	TU Wien\\
	Vienna, Austria \\
	\And
	Joachim Schöberl\\
	Institute of Analysis and Scientific Computing\\
	TU Wien\\
	Vienna, Austria \\
}

\renewcommand{\headeright}{}
\renewcommand{\undertitle}{}

\begin{document}

\maketitle
%\large

\begin{abstract}
In this work, a novel analysis of a hybrid discontinuous Galerkin method for the Helmholtz equation is presented. 
It uses wavenumber, mesh size and polynomial degree independent stabilisation parameters leading to impedance traces between elements.
With analysis techniques based on projection operators unique discrete solvability without a resolution condition and optimal convergence rates with respect to the mesh size are proven. 
The considered method is tailored towards enabling static condensation and the usage of iterative solvers.
\end{abstract}

% keywords can be removed

MSC class: 65N12 (Primary) 65N30 (Secondary)

\section{Introduction}
The numerical solution of the Helmholtz equation with impedance boundary conditions is challenging.
For large wavenumbers the finite element method (FEM) yields, due to the necessary small mesh size, large systems of linear equations.
The computational costs for direct solvers are too expensive and due to the indefinite structure of the Helmholtz equation,
iterative solvers developed for elliptic problems cannot be applied.
Domain decomposition preconditioners seem promising, but sub-blocks of finite element (FE) system matrices need to be inverted.
For conforming FE spaces these blocks represent local Dirichlet problems which may be singular.

Discontinuous spaces can be applied to circumvent this issue leading to absolutely stable local methods, e.g. \cite{feng2013absolutely}.
In \cite{doi:10.1007/s10915-013-9726-8} a discontinuous Galerkin (DG) method for the Helmholtz equation has been proposed and analysed.
It is remarkable that therein the existence of a discrete solution is proven without the need for a resolution condition of the discrete space.
Quasi-optimality has been established in the asymptotic regime.
Other studies of DG methods have been dedicated to giving explicit pre-asymptotic estimates, e.g. \cite{feng2009discontinuous, doi:10.1090/S0025-5718-2011-02475-0, feng2013absolutely, wu2014pre}.
The analysis is based on carefully chosen test functions contained in DG spaces.
The same technique cannot be applied to conforming FE spaces.

A disadvantage of DG methods compared to standard FEM is the larger number of coupling unknowns.
Hybrid discontinuous Galerkin (HDG) methods have been developed to counteract this issue.
Applying static condensation leads to smaller systems of linear equations for skeleton unknowns.
A requirement is the invertibility of element matrices which is closely related to local absolute stability.
In \cite{maxwellabsolute} an absolutely stable HDG method for Maxwell's equations has been analysed
via a new method based on $\ltwo$-projections proving discrete absolute stability.

DG and HDG methods depend on stabilisations of jumps over facets, which usually are $h$-dependent stemming from the analysis of the elliptic Poison equation.
Contrary to that approach, the experiments in \cite{huberphd, huberproceedings, huberproceedingshelmholtz} suggests that iterative solvers for the Helmholtz equation
require $h$-independent stabilisations representing impedance traces between sub-domains.
In \cite{griesmaier2011error}, an HDG method with such a stabilisation has been analysed.
The used technique is based on results from \cite{cockburn2010projection} and utilises special projections tailored to the therein employed HDG formulation.
The authors were able to prove optimal convergence rates in the asymptotic range.

The main contribution of this work is the analysis of the HDG formulation introduced in \cite{Monk2010HybridizingRE} 
and further studied in \cite{huberphd, huberproceedings, huberproceedingshelmholtz}.
The authors explored various promising iterative solvers, motivating the favourable properties of this HDG formulation.
Numerical experiments therein suggest for solvers the necessity of a second variable per facet representing the flux,
which is similar to the discontinuous Petrov Galerkin method in \cite{demkowicz2012wavenumber, gopalakrishnan2014dispersive}.
This work establishes the analytical foundation for the HDG method with two variables per facet introduced in \cite{Monk2010HybridizingRE}.
It utilises a mesh size, domain and wave number independent stabilisation which is apriori known for all Helmholtz problems.
In Section \ref{sec::hdgmain}, the HDG formulation is introduced and the main results, consisting of discrete absolute stability and optimal convergence rates, are stated.
The consecutive sections are dedicated to proving these statements, starting with the discrete absolute stability in Section \ref{sec::discreteabsolutestability},
followed by the optimal convergence for the pressure in Section \ref{sec::errorestimatu} and of the flux in Section \ref{sec::sigmabestapprox}.
In Section \ref{sec::numres}, simulations for problems with plane waves as analytical solutions are shown and
to illustrate the potential of an iterative solver, a 2D example with heterogeneous materials and a 3D example with scattering objects are considered.

\section{Hybrid Discontinuous Galerkin Method} \label{sec::hdgmain}

This study considers the following mixed Helmholtz problem with an impedance boundary condition on a bounded Lipschitz domain $\domain \subset \R^d$.

\begin{definition}[Mixed Helmholtz Problem]
For given $\w > 0$ and $g \in \ltwo(\bnd)$, let $(\sigma, u)$ be the solution of the boundary value problem (BVP)
\begin{subequations}
\begin{align}
	j \w \sigma - \grad u = 0 && \text{ in } \domain, \\
	- \diver \sigma + j \w u = 0 && \text{ in } \domain, \\
	\sigma \cdot \n + u = g && \text{ on } \bnd.
\end{align}
\end{subequations}
\end{definition}
The outward pointing normal vector is denoted by $\n$, the imaginary unit is $j := \sqrt{-1}$, $g$ represents a given excitation on $\bnd$ and $\w$ is a prescribed positive wavenumber.
The flux and pressure are symbolised by $\sigma$ and $u$.
The corresponding adjoint problem is defined as follows.
\begin{definition}[Adjoint Mixed Helmholtz Problem] \label{def::adjointproblem}
For given $\w > 0$ and $f \in \ltwo(\domain)$, let $(\phi, w)$ be the solution of the BVP
\begin{subequations}
\begin{align}
	j \w \phi + \grad w = 0 && \text{ in } \domain, \label{eq:adjfirst}\\
	\diver \phi + j \w w = f && \text{ in } \domain, \label{eq:adjsecond}\\
	\phi \cdot \n + w = 0 && \text{ on } \bnd. \label{eq:adjboundary}
\end{align}
\end{subequations}
\end{definition}
The adjoint BVP has a homogenous Robin-boundary condition and a volume excitation $f$.
The existence and uniqueness of the continuous BVPs have been proven in \cite[Proposition 8.1.3]{melenkphd}.

\begin{remark}
The mixed adjoint problem is equivalent to
\begin{subequations}
\begin{align}
	- \laplace w - \w^2 w = j \w f && \text{ in } \domain, \\
	- \frac{1}{j \w} \grad w \cdot \n + w = 0 && \text{ on } \bnd
\end{align}
\end{subequations}
with $\w$-dependent right-hand side.
\end{remark}

To be able to achieve $\w$-explicit analysis, regularity estimates of the Helmholtz equation are necessary.
For convex, or smooth and star-shaped domains, the adjoint Helmholtz problem in Definition \ref{def::adjointproblem} satisfies the regularity
\begin{align}
	\lVert \grad w \rVert_{\ltwo(\domain)} + \w \lVert w \rVert_{\ltwo(\domain)} + \w \lVert \phi \rVert_{\ltwo(\domain)} &\leq C \w \lVert f \rVert_{\ltwo(\domain)}, \\
	\lVert \grad^2 w \rVert_{\ltwo(\domain)} + \w \lVert \grad \phi \rVert_{\ltwo(\domain)} &\leq C (1+\w) \w \lVert f \rVert_{\ltwo(\domain)},
\end{align}
with a constant $C > 0$ independent of $\w$, see \cite[Remark 2.6]{doi:10.1007/s10915-013-9726-8}.
Throughout this study, the $H^2$ regularity of both BVPs is assumed.

The domain is split into disjoint finite elements $T$. The set of all finite elements is called $\mathcal{T}$ and there holds $\cup_{T \in \mathcal{T}} T = \domain$.
The facets between elements inside of $\domain$ are denoted by $F_{\rm I}$ and the facets on $\bnd$ by $F_{\rm O}$.
If a distinction between $F_{\rm{I}}$ and $F_{\rm{O}}$ is not necessary, then a facet will be called $F$ and the set of all facets is defined as $\mathcal{F}$.
The following short notations for $\ltwo$ integrals on volume and facet terms are used,
\begin{align}
	(u, v)_T := \int_T u \cdot \conjugate{v} d \x, && (u, v)_F := \int_F u \cdot \conjugate{v} d s,
\end{align}
and the accordingly induced norms are
\begin{align}
	\lVert u \rVert_T := \sqrt{(u, u)_T}, && \lVert u \rVert_F := \sqrt{(u, u)_F}.
\end{align}
It is noteworthy that the second argument is complex conjugated.
The used discontinuous, complex-valued spaces are defined by:
\begin{align}
	\hdivel &:= \{\sigma \in [\ltwo(T)]^d : \diver \sigma \in \ltwo(T)  \}, \\
	\hdiv &:= \prod_{T \in \mathcal{T}}\hdivel, \\
	\ltwo(\mathcal{F}) &:= \prod_{F \in \mathcal{F}} \ltwo(F).
\end{align}
The respective discrete spaces, of the polynomial degree $p \in \N_0$, are
\begin{align}
	\mathcal{RT}^p(T) &:= \{\sigma + \x \ q: \sigma \in [\mathcal{P}^p(T)]^d , q \in \mathcal{P}^p(T)\} \subset \hdivel, \\
	\mathcal{RT}^p (\mathcal{T}) &:= \prod_{T \in \mathcal{T}} \mathcal{RT}^p(T) \subset \hdiv, \\
	\mathcal{P}^p(\mathcal{F}) &:= \prod_{F \in \mathcal{F}} \mathcal{P}^p(F) \subset \ltwo(\mathcal{F}), \\
	\mathcal{P}^p(\mathcal{T}) &:= \prod_{T \in \mathcal{T}} \mathcal{P}^p(T) \subset \ltwo(\domain),
\end{align}
where $\x := (x,y,z)^\top$ and $\mathcal{P}^p$ is the space of polynomials with degrees smaller or equal to $p$.
For the discrete, discontinuous compound space the short notation
\begin{align}
	\mathcal{X}_h := \mathcal{RT}^p (\mathcal{T}) \times \mathcal{P}^p(\mathcal{T}) \times \mathcal{P}^p(\mathcal{F}) \times \mathcal{P}^p(\mathcal{F})
\end{align}
is used. The discrete space $\mathcal{RT}$ is commonly known as the Raviart-Thomas space, see \cite{SchoeZagl:05}.

The consistent HDG formulation introduced in \cite{Monk2010HybridizingRE} and considered in this work is defined as follows.

\begin{definition}[HDG Formulation] \label{def::hdgformulation}
For given $g \in \ltwo(\bnd)$ and $\w > 0$ find $\sigma \in \hdiv$, $u \in \ltwo(\domain)$, $\hat u \in \ltwo(\mathcal{F})$, $\hat \sigma_\n \in \ltwo(\mathcal{F})$ such that
\begin{align}
\begin{split}
	B(\sigma, \hat \sigma_\n, u, \hat u, \tau, \hat \tau_\n, v, \hat v) = - (g, \hat v)_\bnd \label{eq::weakform}
\end{split}
\end{align}
holds for all $\tau \in \hdiv$, $v \in \ltwo(\domain)$, $\hat v \in \ltwo(\mathcal{F})$, $\hat \tau_\n \in \ltwo(\mathcal{F})$
with the sesquilinear form (SLF)
\begin{align}
\begin{split}
	B(\sigma, \hat \sigma_\n, u, \hat u, \tau, \hat \tau_\n, v, \hat v) :=
	&\sum_{T \in \mathcal{T}} \Big ( j \w (\sigma, \tau)_T + (u, \diver \tau)_T + (\diver \sigma, v)_T - j \w (u, v)_T \\
	&- (\hat u, \tau \cdot \n)_\elbnd - (\sigma \cdot \n, \hat v)_\elbnd - \alpha ([u], [v])_\elbnd + \beta ( \llbracket \sigma \rrbracket, \llbracket \tau \rrbracket)_\elbnd \Big ) - (\hat u, \hat v)_\bnd.
\end{split}
\end{align}
The jumps are defined as
\begin{align}
	[u] := u - \hat u, && \llbracket \sigma \rrbracket := \sigma \cdot \n - \hat \sigma_{\n}, && \text{ on } \elbnd
\end{align}
and the stabilisation parameters are $\alpha := 1$, $\beta := 1$.
\end{definition}
The normal vector $\n$ is per default outward-oriented on $T$
and the use as the subscript for the facet variables $\hat \sigma_\n$ and $\hat \tau_\n$ indicates that the variables are scalar-valued and their sign changes with respect to the adjacent elements.
\begin{remark}
An interesting property of the HDG formulation is the choice of the $h$-independent parameters $\alpha$ and $\beta$.
The reason for the choices in this study is due to the favourable properties for iterative solvers and their preconditioning established in \cite{huberphd, huberproceedings, huberproceedingshelmholtz}.
\end{remark}

For methods with static condensation, the sub-problems defined on individual elements must be uniquely and stably solvable.
The HDG formulation \eqref{eq::weakform} satisfies that condition and the following discrete absolute stability result holds.

\begin{theorem} \label{th::stability}
Assuming $H^2$-regularity of the adjoint problem, there exists a constant $\tilde C > 0$, independent of $\w, h, \alpha, \beta$,
so that there holds for the discrete solution of the HDG formulation \eqref{eq::weakform}
\begin{align}
	\sum_{T \in \mathcal{T}} \alpha \lVert [u_h] \rVert^2_{\ltwo(\elbnd)} + \beta \lVert \llbracket \sigma_h \rrbracket \rVert^2_{\ltwo(\elbnd)} + \lVert \hat u_h \rVert^2_{\ltwo(\bnd)} &\leq \lVert g \rVert^2_{\ltwo(\bnd)}, \\
	\sum_{T \in \mathcal{T}} \w \lVert \sigma_h \rVert^2_{\ltwo(T)} &\leq \sum_{T \in \mathcal{T}} \w \lVert u_h \rVert^2_{\ltwo(T)} + \lVert g \rVert^2_{\ltwo(\bnd)}, \\
	\lVert u_h \rVert_{\ltwo(\domain)} &\leq C(\w, h, \alpha, \beta) \lVert g \rVert_{\ltwo(\bnd)},
\end{align}
with the constant
\begin{align}
	C(\w, h,\alpha, \beta) = \tilde C
		\Bigg ( \left (\frac{2}{\alpha} + \beta \right ) (1+\w)^2 h
			+ 2 \alpha (1 + \w)^2 \w^2 h^3
			+ 1 \Bigg )^{\frac{1}{2}}.
\end{align}
\end{theorem}

This theorem implies the existence and uniqueness of the discrete solution for the HDG formulation, without a resolution condition on the discrete space.
The proof uses a similar approach as introduced in \cite{maxwellabsolute} for Maxwell's equations. 
Two major differences are that only a single facet space is used in the mentioned work and the stabilisation parameter therein is $h$-dependent.
The method in \cite{maxwellabsolute} is related to the case of choosing $\alpha \approx \frac{1}{h}$, $\beta = 0$ and omitting the second facet variable in the HDG formulation \eqref{eq::weakform}.
The main idea of the proof is to use an Aubin-Nitsche technique along with $\ltwo$-projections of $\sigma$ and $u$ so that volume terms vanish and only facet terms remain.
The real part of the HDG formulation controls those facet terms.
With similar techniques and arguments, the following optimal convergence rates for jumps and the pressure are proven.
\begin{theorem} \label{th::faceterror}
Assuming $H^{p+2}$-regularity of the Helmholtz problem, there exists a constant $C > 0$, independent of $\w, h, \alpha, \beta$, such that
\begin{align}
\begin{split}
	\sum_{T \in \mathcal{T}} \alpha \lVert [u - u_h] \rVert_{\ltwo(\elbnd)}^2 + \beta \Vert \llbracket \sigma - \sigma_h \rrbracket \rVert_{\ltwo(\elbnd)}^2 + \lVert u &- \hat u_h \rVert_{\ltwo(\bnd)}^2 \\
		&\leq C h^{2p + 1} \Bigg ( \left (\frac{2}{\alpha} + \beta \right )  \lVert \sigma \rVert_{H^{p+1}(\domain)}^2 + 2 \alpha \lVert u \rVert_{H^{p+1}(\domain)}^2 \Bigg ).
\end{split}
\end{align}
\end{theorem}

\begin{theorem} \label{th::scalarerror}
Assuming $H^2$-regularity of the adjoint problem and $H^{p+2}$-regularity of the Helmholtz problem,
there exists a constant $C > 0$, independent of $\w, h, \alpha, \beta$, such that
\begin{align}
\begin{split}
	\lVert u - &u_h \rVert_{\ltwo(\domain)} \\
	&\leq C h^{p + 1} \left ( \left (\frac{2}{\alpha} + \beta \right ) \left ( 1 + \w \right )^2 + 2 \alpha (1 + \w)^2 \w^2 h^{2} \right)^{\frac{1}{2}}
		\left ( \left (\frac{2}{\alpha} + \beta \right ) \lVert \sigma \rVert_{H^{p+1}(\domain)}^2 + 2 \alpha  \lVert u \rVert_{H^{p+1}(\domain)}^2 \right )^{\frac{1}{2}}.
\end{split}
\end{align}
\end{theorem}

Due to the $h$-independent stabilisation, the optimal convergence rate for the flux cannot be straightforwardly derived.
For the following result, a more refined technique based on projections is needed.

\begin{theorem} \label{th::sigmaerror}
Assuming $H^2$-regularity of the adjoint problem and $H^{p+2}$-regularity of the Helmholtz problem,
there exists a constant $C(\w, \alpha, \beta) > 0$, independent of $h$, such that
\begin{align}
	\lVert \sigma - \sigma_h \rVert_{\ltwo{(\domain)}} \leq C(\w, \alpha, \beta) h^{p+1} \left (\lVert \sigma \rVert_{H^{p+1}(\domain)} + \lVert u \rVert_{H^{p+1}(\domain)} \right).
\end{align}
\end{theorem}
The explicit form of $C(\w, \alpha, \beta)$ can be seen in Section \ref{sec::sigmabestapprox}.
The contribution of \cite{griesmaier2011error}, which is based upon the techniques developed in \cite{cockburn2010projection}, motivates the proof of this theorem.
A comprehensive and thorough explanation of the later work can be found in \cite{Sayas2013FromRT}.
In these studies, HDG methods with $h$-independent stabilisation are analysed. 
The HDG formulation in \cite{griesmaier2011error} is similar to the formulation in this work.
The first major difference is that only a single facet variable with a comparable $\alpha$-stabilisation is used
and the second is that existence, uniqueness and optimal convergence rates are proven under the assumption of a resolution condition.
Proving optimal rates for $h$-independently stabilised HDG formulation is challenging, 
because element boundary terms need to be estimated via an inverse estimate with volume terms leading to sub-optimal rates.
The idea is to use a, to the HDG formulation tailored, projection into the discrete space, so that boundary terms vanish.
For that purpose, a special projection has been established and analysed in \cite{cockburn2010projection}.
In this study, due to the second $\beta$-stabilisation, that projection cannot be applied.
A generalisation needs to be established for the proof.

\section{Discrete Absolute Stability} \label{sec::discreteabsolutestability}

Usually, for finite element discretisations of the Helmholtz equation, a resolution condition is required to prove the existence and uniqueness of discrete solutions as well as quasi-optimality.
These results hold in the asymptotic regime. Several publications have established the existence and uniqueness for DG and HDG methods without such a condition,
e.g. \cite{feng2009discontinuous, doi:10.1090/S0025-5718-2011-02475-0, feng2013absolutely, doi:10.1007/s10915-013-9726-8, moiola2014helmholtz, wu2014pre, FengLuXu+2016+429+445}.
The idea is to use the discrete test function $\x \cdot \grad u_h$ which is contained in the discrete discontinuous space and leads to positive volume terms.
The disadvantage is that the pre-asymptotic and the asymptotic analysis require different techniques.

A new approach based upon $\ltwo$-projections, which leads to results viable in the pre-asymptotic as well as the asymptotic regime, has been developed in \cite{maxwellabsolute}.
In this section, the technique therein is applied to the HDG formulation \eqref{eq::weakform}, with the additional consideration of a second facet variable and $\beta$-stabilisation as well as $h$-independent $\alpha$-stabilisation.

The following lemma establishes the consistency of the adjoint Helmholtz problem in Definition \ref{def::adjointproblem}.

\begin{lemma}[Weak Adjoint Formulation] \label{lem::weakadjoint}
The adjoint solution $(\phi, w)$ of \eqref{eq:adjfirst} - \eqref{eq:adjboundary},  satisfies
\begin{align}
	B( \tau_h, \hat \tau_{\n,h}, v_h, \hat v_h, \phi, \phi \cdot \n, w, w) = (v_h, f)_\domain,
\end{align}
\end{lemma}
for all $(\tau_h, \hat \tau_{\n,h}, v_h, \hat v_h) \in \mathcal{X}_h$.
\begin{proof}
Conjugating \eqref{eq:adjfirst} and \eqref{eq:adjsecond} leads to
\begin{align}
	- j \w \conjugate \phi + \grad \conjugate w &= 0, \\
	- \diver \conjugate \phi + j \w \conjugate w &= - \conjugate f,
\end{align}
further multiplying by $\tau_h$, $v_h$ respectively and integrating over $T$ gives
\begin{align}
	- j \w (\tau_h, \phi)_T + (\tau_h,  \grad w)_T &= 0, \\
	- (v_h, \diver \phi)_T + j \w (v_h, w)_T &= - (v_h, f)_T.
\end{align}
Element wise partial integration of $(\tau_h,  \grad w)_T$ and summation over all elements results in
\begin{align}
	\sum_{T \in \mathcal{T}} j \w (\tau_h, \phi)_T + (\diver \tau_h, w)_T + (v_h, \diver \phi)_T - j \w (v_h, w)_T 
	- (\tau_h \cdot \n, w)_\elbnd = (v_h, f)_\domain.
\end{align}
Whenever jumps between variables on adjacent elements over facets are considered, then one side is marked by the $+$-subscript and the other by $-$.
The normal continuity of $\phi$ on $F_{\rm I}$ implies
\begin{align}
	0 = - (\hat v_h, \phi_+ \cdot \n_+ - \phi_- \cdot \n_+)_{F_{\rm I}} = - (\hat v_h, \phi_+ \cdot \n_+)_{F_{\rm I}} - (\hat v_h, \phi_- \cdot \n_-)_{F_{\rm I}}
\end{align}
and due to the boundary condition \eqref{eq:adjboundary} there holds 
\begin{align}
	0 = -(\hat v_h, \phi \cdot \n)_\bnd - (\hat v_h, w)_\bnd.
\end{align}
Adding these terms and considering that the $\alpha$- and $\beta$-stabilisations vanish due to the continuity of $\phi \cdot \n$ and $w$ on element boundaries concludes the proof.
\end{proof}

The following first stability results follow by considering the real ($\Re$) and imaginary ($\Im$) part of the SLF $B$ separately.
This lemma and the proof of it are similar to \cite[Lemma 3.1]{maxwellabsolute}.

\begin{lemma} \label{lam::bounds}
For the discrete solution of \eqref{eq::weakform} there holds
\begin{align}
	\sum_{T \in \mathcal{T}} \alpha \lVert [u_h] \rVert^2_\elbnd + \beta \lVert \llbracket \sigma_h \rrbracket \rVert^2_\elbnd + \lVert \hat u_h \rVert^2_\bnd &\leq \lVert g \rVert^2_\bnd, \label{eq::bndstabestimate}\\
	\sum_{T \in \mathcal{T}} \w \lVert \sigma_h \rVert^2_T &\leq \sum_{T \in \mathcal{T}} \w \lVert u_h \rVert^2_T + \lVert g \rVert^2_\bnd. \label{eq::volstabestimate}
\end{align}
\end{lemma}
\begin{proof}
By choosing $v_h = - u_h, \tau_h = \sigma_h, \hat v_h = - \hat u_h, \hat \tau_{n,h} = \hat \sigma_{n,h}$ as test functions $B$ has the form
\begin{align}
\begin{split}
	B(\sigma_h, \hat \sigma_{\n,h},  u_h, \hat u_h, \sigma_h, \hat \sigma_{\n,h},  -u_h, -\hat u_h) =&
	\sum_{T \in \mathcal{T}} j \w \lVert \sigma_h \rVert^2_T + 2 \Im (u_h, \diver \sigma_h)_T + j \w \lVert u_h \rVert^2_T \\
	&- 2 \Im (\hat u_h, \sigma_h \cdot \n)_\elbnd
	 + \alpha \lVert [u_h] \rVert^2_\elbnd + \beta \lVert \llbracket \sigma_h \rrbracket \rVert^2_\elbnd 
	 + \lVert \hat u_h \rVert^2_\bnd.
\end{split}
\end{align}
Taking the real part leads to
\begin{align}
	\alpha \lVert [u_h] \rVert^2_\elbnd + \beta \lVert \llbracket \sigma_h \rrbracket \rVert^2_\elbnd + \lVert \hat u_h \rVert^2_\bnd = \Re (g, \hat u_h)_\bnd
\end{align}
directly implying \eqref{eq::bndstabestimate}.
Similarly, choosing $v_h = u_h, \tau_h = \sigma_h, \hat v_h = \hat u_h, \hat \tau_{n,h} = \hat \sigma_{n,h}$ as test functions gives
\begin{align}
\begin{split}
	B(\sigma_h, \hat \sigma_{\n,h},  u_h, \hat u_h, \sigma_h, \hat \sigma_{\n,h},  u_h, \hat u_h) =&
	\sum_{T \in \mathcal{T}} j \w \lVert \sigma_h \rVert^2_T + 2 \Re (u_h, \diver \sigma_h)_T - j \w \lVert u_h \rVert^2_T \\ 
	&- 2 \Re (\hat u_h, \sigma_h \cdot \n)_\elbnd
	- \alpha \lVert [u_h] \rVert^2_\elbnd + \beta \lVert \llbracket \sigma_h \rrbracket \rVert^2_\elbnd
	- \lVert \hat u_h \rVert^2_\bnd.
\end{split}
\end{align}
Taking the imaginary part yields
\begin{align}
	\sum_{T \in \mathcal{T}} \w \lVert \sigma_h \rVert^2_T - \w \lVert u_h \rVert^2_T = - \Im (g, \hat u_h)_\bnd,
\end{align}
and with \eqref{eq::bndstabestimate}, which implies $\lVert \hat u_h \rVert_\bnd \leq \lVert g \rVert_\bnd$, gives \eqref{eq::volstabestimate}.
\end{proof}

Lemma \ref{lam::bounds} shows that $\hat u_h$ and the discrete jumps are bounded by the excitation,
and if $u_h$ can be bounded then the same follows for $\sigma_h$. 
The result highlights the favourable choice of real, positive $\alpha$ and $\beta$.
The adjoint problem in combination with an Aubin-Nitsche type technique is used to bound $u_h$.

\begin{lemma} \label{lem::trace}
For the solution $(\phi, w)$ of \eqref{eq:adjfirst} - \eqref{eq:adjboundary} there holds
\begin{align}
	\lVert w \rVert^2_\bnd = \lVert \phi \cdot \n \rVert^2_\bnd \leq C \lVert f \rVert^2_\domain,
\end{align}
with a constant $C > 0$ independent of $\w$.
\end{lemma}
\begin{proof}
The first equality immediately follows from \eqref{eq:adjboundary}.
Multiplying \eqref{eq:adjfirst} and \eqref{eq:adjsecond} by $\conjugate \phi$ and $\conjugate w$ respectively gives
\begin{align}
	j \w \lVert \phi \rVert^2_\domain + (\grad w, \phi)_\domain &= 0, \\
	j \w \lVert w \rVert^2_\domain + (\diver \phi, w)_\domain &= (f, w)_\domain.
\end{align}
Partial integration of $(\grad w, \phi)_\domain$ leads to
\begin{align}
	j \w \lVert \phi \rVert^2_\domain - (w, \diver \phi)_\domain + (w, \phi \cdot \n)_\bnd = 0.
\end{align}
By conjugating and adding the equations the terms $(\diver \phi, w)_\domain$ cancel out and
\begin{align}
	j \w \lVert w \rVert^2_\domain - j \w \lVert \phi \rVert^2_\domain + (\phi \cdot \n, w)_\bnd = (f, w)_\domain
\end{align}
remains.
Using \eqref{eq:adjboundary} and only considering the real part in combination with the adjoint regularity implies
\begin{align}
	\lVert \phi \cdot \n \rVert^2_\bnd \leq \frac{1}{\w} \lVert f \rVert_\domain \w \lVert w \rVert_\domain \leq C \lVert f \rVert^2_\domain.
\end{align}
\end{proof}

The following lemma is essential for pre-asymptotic stability. It highlights the effect of introducing $\ltwo$-projections of $\sigma$ and $u$ defined by
\begin{align}
	(\Pi \sigma, \tau_h)_T = (\sigma, \tau_h)_T && \forall \tau_h \in \mathcal{RT}^p(T), \\
	(\Pi u, v_h)_T = (u, v_h)_T && \forall v_h \in \mathcal{P}^p(T), \\
	(\Pi_F u, \hat v_h)_F = (u, \hat v_h)_F && \forall \hat v_h \in \mathcal{P}^p(F),
\end{align}
into the SLF. The proof follows the lines in \cite[Lemma 3.2]{maxwellabsolute}.

\begin{lemma} \label{lem::adjprojectproperty}
Using the $\ltwo$-projection $\Pi$ on elements and the $\ltwo$-projection $\Pi_F$ on facets, there holds for arbitrary $\tau_h, \hat \tau_{\n, h}, v_h, \hat v_h \in \mathcal{X}_h$
\begin{align}
\begin{split}
	B( \tau_h, \hat \tau_{\n,h}, v_h, \hat v_h, \phi - \Pi \phi, &\phi \cdot \n - \Pi_F \phi \cdot \n, w - \Pi w, w - \Pi_F w) \\
	= \sum_{T \in \mathcal{T}}&([v_h], (\phi - \Pi \phi) \cdot \n)_\elbnd - \alpha ([v_h], w - \Pi w)_\elbnd + \beta (\llbracket \tau_h \rrbracket, (\phi - \Pi \phi) \cdot \n)_\elbnd.
\end{split}
\end{align}
\end{lemma}
\begin{proof}
The expression on the left-hand side of the equation is
\begin{align}
\begin{split}
	B( \tau_h, \hat \tau_{\n,h}, v_h, \hat v_h,& \phi - \Pi \phi, \phi \cdot \n - \Pi_F \phi \cdot \n, w - \Pi w, w - \Pi_F w) \\
	=& \sum_{T \in \mathcal{T}} j \w (\tau_h, \phi - \Pi \phi)_T + (v_h, \diver (\phi - \Pi \phi))_T + (\diver \tau_h, w - \Pi w)_T - j \w (v_h, w - \Pi w)_T \\
	&- (\hat v_h, (\phi - \Pi \phi) \cdot \n)_\elbnd - (\tau_h \cdot \n, w - \Pi_F w)_\elbnd \\
	&- \alpha ([v_h], [w - \Pi w])_\elbnd + \beta (\llbracket \tau_h \rrbracket, \llbracket \phi - \Pi \phi \rrbracket)_\elbnd - (\hat v_h, w - \Pi_F w)_\bnd.
\end{split}
\end{align}
Due to the $\ltwo$-projection properties the terms
\begin{align}
	j \w (\tau_h, \phi - \Pi \phi)_T = 0, && (\diver \tau_h, w - \Pi w)_T = 0, && - j\w (v_h, w - \Pi w)_T = 0, \\
	- (\tau_h \cdot \n, w - \Pi_F w)_\elbnd = 0, && - (\hat v, w - \Pi_F w)_\bnd = 0
\end{align}
vanish.
The remaining volume term introduces, after partial integration, the following element boundary terms
\begin{align}
	(v_h, \diver (\phi - \Pi \phi))_T = - (\grad v_h, \phi - \Pi \phi)_T + (v_h, (\phi - \Pi \phi) \cdot \n)_\elbnd = (v_h, (\phi - \Pi \phi) \cdot \n)_\elbnd
\end{align}
and for jumps there hold
\begin{align}
	- \alpha ([v_h], [w - \Pi w])_\elbnd &= - \alpha ([v_h], w - \Pi w - w + \Pi_F w)_\elbnd =- \alpha ([v_h], w - \Pi w)_\elbnd, \\
	\beta (\llbracket \tau_h \rrbracket, \llbracket \phi - \Pi \phi \rrbracket)_\elbnd &= \beta (\llbracket \tau_h \rrbracket, (\phi - \Pi \phi) \cdot \n)_\elbnd.
\end{align}
Incorporating these changes yields
\begin{align}
\begin{split}
	B( \tau_h, \hat \tau_{\n,h}, v_h, &\hat v_h, \phi - \Pi \phi, \phi_\n - \Pi_F \phi_\n, w - \Pi w, w - \Pi_F w) \\
	= \sum_{T \in \mathcal{T}}&(v_h - \hat v_h, (\phi - \Pi \phi) \cdot \n)_\elbnd
		- \alpha ([v_h], w - \Pi w)_\elbnd + \beta (\llbracket \tau_h \rrbracket, (\phi - \Pi \phi) \cdot \n)_\elbnd.
\end{split}
\end{align}
\end{proof}

To prove the main result of this section the following standard approximation properties of $\ltwo$-projections are required.

\begin{lemma}[Approximation Properties of $\ltwo$-projections]
Assuming $H^{s+1}$-regularity of $\phi$, $H^{t+1}$-regularity of $w$, with $s,t \in \R_+$. If the polynomial degree of the discrete spaces satisfies $p \geq s \geq 0$, $p \geq t \geq 0$,
then exists a constant $C > 0$ independent of $h$ so that
\begin{align}
	\sum_{T \in \mathcal{T}} \lVert (\phi - \Pi \phi) \cdot \n \rVert_\elbnd^2 \leq C h^{2s + 1} \lVert \phi \rVert_{H^{s+1}}^2, \\
	\sum_{T \in \mathcal{T}} \lVert w - \Pi w \rVert_\elbnd^2 \leq C h^{2t + 1} \lVert w \rVert_{H^{t+1}}^2.
\end{align}
\end{lemma}

In the following theorem $u_h$ will be bounded by $g$, which concludes the stability analysis of the HDG formulation.
The proof is based upon \cite[Lemma 3.2]{maxwellabsolute}.

\begin{proof}[Proof of Theorem \ref{th::stability}]
Using an Aubin-Nitsche trick for the adjoint problem, by considering as excitation $f = u_h$, yields in combination with Lemma \ref{lem::weakadjoint}
\begin{align}
\begin{split}
	\lVert u_h \rVert^2_\domain = B( \sigma_h, \hat \sigma_{\n,h}, u_h, \hat u_h, \phi, \phi_\n, w, w),
\end{split}
\end{align}
with $(\sigma_h, \hat \sigma_{\n,h}, u_h, \hat u_h)$ as test functions.
Adding and subtracting the $\ltwo$-projections $\Pi$ and $\Pi_F$ leads to
\begin{align}
\begin{split}
	\lVert u_h \rVert^2_\domain =\ & B( \sigma_h, \hat \sigma_{\n,h}, u_h, \hat u_h, \phi - \Pi \phi, \phi_\n - \Pi_F \phi_\n, w - \Pi w, w - \Pi_F w) \\
		&+ B( \sigma_h, \hat \sigma_{\n,h}, u_h, \hat u_h, \Pi \phi, \Pi_F \phi_\n, \Pi w, \Pi_F w).
\end{split}
\end{align}
The first part is replaced by Lemma \ref{lem::adjprojectproperty}
and according to \eqref{eq::weakform} there holds for the second part
\begin{align}
	B( \sigma_h, \hat \sigma_{\n,h}, u_h, \hat u_h, \Pi \phi, \Pi_F \phi_\n, \Pi w, \Pi_F w) = - (g, \Pi_F w)_\bnd,
\end{align}
because the projected adjoint solution is in the discrete space $\mathcal{X}_h$.
Cauchy-Schwarz and Lemma \ref{lam::bounds} give the estimate
\begin{align}
\begin{split}
	\lVert u_h \rVert^2_\domain
	=& \sum_{T \in \mathcal{T}} ([u_h], (\phi - \Pi \phi) \cdot \n)_\elbnd - \alpha ([u_h], w - \Pi w)_\elbnd + \beta (\llbracket \sigma_h \rrbracket, (\phi - \Pi \phi) \cdot \n)_\elbnd - (g, \Pi_F w)_\bnd \\
	\leq& \Bigg ( \sum_{T \in \mathcal{T}} \alpha \lVert  [u_h] \rVert_\elbnd^2 
		+ \beta \lVert \llbracket \sigma_h \rrbracket \rVert_\elbnd^2
		+ \lVert g \rVert_\bnd^2 \Bigg )^{\frac{1}{2}}\\
	&\Bigg ( \sum_{T \in \mathcal{T}} \frac{2}{\alpha} \lVert (\phi - \Pi \phi) \cdot \n \rVert_\elbnd^2
		+ 2\alpha \lVert w - \Pi w \rVert_\elbnd^2
		+ \beta \lVert ( \phi - \Pi \phi ) \cdot \n \rVert_\elbnd^2
		+ \lVert \Pi_F w \rVert_\bnd^2 \Bigg )^{\frac{1}{2}} \\
	\leq& \sqrt{2} \lVert g \rVert_\bnd
		\Bigg ( \sum_{T \in \mathcal{T}} \frac{2}{\alpha} \lVert (\phi - \Pi \phi) \cdot \n \rVert_\elbnd^2
		+ 2\alpha \lVert w - \Pi w \rVert_\elbnd^2
		+ \beta \lVert ( \phi - \Pi \phi ) \cdot \n \rVert_\elbnd^2
		+ \lVert \Pi_F w \rVert_\bnd^2 \Bigg )^{\frac{1}{2}},
\end{split}
\end{align}
and further applying the $\ltwo$ approximation properties and Lemma \ref{lem::trace} for the term on the domain boundary yields
\begin{align}
\begin{split}
	\lVert u_h \rVert^2_\domain 
	&\leq C \lVert g \rVert_\bnd 
		\Bigg ( \left (\frac{2}{\alpha} + \beta \right ) h^{2s+1} \lVert \phi \rVert_{H^{s+1}}^2
			+ 2 \alpha h^{2t + 1} \lVert w \rVert_{H^{t+1}}^2
			+ \lVert u_h \rVert^2_\domain \Bigg )^{\frac{1}{2}}.
\end{split}
\end{align}
Due to the $H^2$-regularity of the adjoint problem $s = 0, t = 1$ can be chosen leading to the stability constant
\begin{align}
	C(\w, h,\alpha, \beta) = C
		\Bigg ( \left (\frac{2}{\alpha} + \beta \right ) \left (1+\w \right )^2 h
			+ 2 \alpha (1 + \w)^2 \w^2 h^3
			+ 1 \Bigg )^{\frac{1}{2}}.
\end{align}
\end{proof}

\section{Error Estimates for the Pressure and Jumps} \label{sec::errorestimatu}

In this section Theorem \ref{th::faceterror} and Theorem \ref{th::scalarerror} are proven. The analysis is similar to the proof of stability in the last section and leans on \cite{maxwellabsolute}.
The continuous and discrete solutions satisfy
\begin{align}
\begin{split}
	B(\sigma, \sigma_\n, u, u, \tau_h, \hat \tau_{\n,h}, v_h, \hat v_h) &= - (g, \hat v_h)_\bnd, \\
	B(\sigma_h, \hat \sigma_{\n,h}, u_h, \hat u_h, \tau_h, \hat \tau_{\n,h}, v_h, \hat v_h) &= - (g, \hat v_h)_\bnd,
\end{split}
\end{align}
for all $(\tau_h, \hat \tau_{\n,h}, v_h, \hat v_h) \in \mathcal{X}_h$.
In the case of coercive weak formulations optimal convergence rates are proven with coercivity, Galerkin orthogonality and continuity.
The Helmholtz equation does not satisfy coercivity, therefore other techniques need to be established.
Usually, a Schatz argument is applied to derive asymptotic results, but these depend on a resolution condition.
The $\ltwo$-projections circumvent this necessity.

For the analysis, the following projected errors, contained in the discontinuous space $\mathcal{X}_h$, are needed.
\begin{definition}
\begin{align}
	e_\sigma := \Pi \sigma - \sigma_h, && e_u := \Pi u - u_h, && e_{\hat u} := \Pi_F u - \hat u_h, && e_{\hat \sigma} := \Pi_F \sigma \cdot \n - \hat \sigma_{\n,h}. \label{eq::projectederrors}
\end{align}
\end{definition}
Projections lead to a splitting of error estimates into e.g.
\begin{align}
	\lVert u - u_h \rVert_\domain = \lVert u - \Pi u + \Pi u - u_h \rVert_\domain \leq \lVert u - \Pi u \rVert_\domain + \lVert \Pi u - u_h \rVert_\domain.
\end{align}
The first part only depends on the approximation properties of the projection. 
The second part holds the advantage that $\Pi u - u_h$ is a viable choice as a discrete test function.
Therefore, if the projection has optimal approximation properties and if the projected error has an optimal convergence rate, then the discrete solution has an optimal rate as well.

For the projected errors the Galerkin orthogonality does not hold, but they satisfy the following discrete weak formulation.

\begin{lemma} \label{lem::errorweak}
The projected errors in \eqref{eq::projectederrors} satisfy the weak formulation
\begin{align}
\begin{split}
	B(e_\sigma, e_{\hat \sigma}, e_u, e_{\hat u}, \tau_h, \hat \tau_{\n,h}, &v_h, \hat v_h) \\
	=& - \sum_{T \in \mathcal{T}} (\sigma \cdot \n - \Pi \sigma \cdot \n, [v_h])_\elbnd - \alpha (u - \Pi u, [v_h])_\elbnd + \beta (\sigma \cdot \n - \Pi \sigma \cdot \n, \llbracket \tau_h \rrbracket)_\elbnd,
\end{split}
\end{align}
for all $(\tau_h, \hat \tau_{\n,h}, v_h, \hat v_h) \in \mathcal{X}_h$.
\end{lemma}
\begin{proof}
By testing with discrete functions, continuous variables can be replaced with their $\Pi_F$-projections on facets. 
The volume projection $\Pi$ can be inserted into most parts, except for
\begin{align}
\begin{split}
	\sum_{T \in \mathcal{T}} (\diver \sigma, v_h)_T 
	&= \sum_{T \in \mathcal{T}} - (\sigma, \grad v_h)_T + (\Pi_F \sigma \cdot \n, v_h)_\elbnd \\
	&= \sum_{T \in \mathcal{T}} - (\Pi \sigma, \grad v_h)_T + (\Pi_F \sigma \cdot \n, v_h)_\elbnd \\
	&= \sum_{T \in \mathcal{T}} (\diver \Pi \sigma, v_h)_T + (\Pi_F \sigma \cdot \n - \Pi \sigma \cdot \n, v_h)_\elbnd,
\end{split}
\end{align}
which results in boundary terms.
Taking them into account yields
\begin{align}
\begin{split}
	B(\sigma,\ \sigma_\n, &u, u, \tau_h, \hat \tau_{\n,h}, v_h, \hat v_h) \\
	=& \sum_{T \in \mathcal{T}} j \w (\Pi \sigma, \tau_h)_T + (\Pi u, \diver \tau_h)_T + (\diver \Pi \sigma, v_h)_T - j \w (\Pi u, v_h)_T + (\Pi_F \sigma \cdot \n - \Pi \sigma \cdot \n, v_h)_\elbnd \\
	&- (\Pi_F u, \tau_h \cdot \n)_\elbnd - (\Pi_F \sigma \cdot \n, \hat v_h)_\elbnd - \alpha (\Pi_F [u], [v_h])_\elbnd 
		+ \beta (\Pi_F \llbracket \sigma \rrbracket, \llbracket \tau_h \rrbracket)_\elbnd - (\Pi_F u, \hat v_h)_\bnd,
\end{split}
\end{align}
which leads to
\begin{align}
\begin{split}
	0 =\ & B(\sigma, \sigma_\n, u, u, \tau_h, \hat \tau_{\n,h}, v_h, \hat v_h) - B(\sigma_h, \hat \sigma_{\n,h}, u_h, \hat u_h, \tau_h, \hat \tau_{\n,h}, v_h, \hat v_h) \\
	=\ & \sum_{T \in \mathcal{T}} j \w (e_\sigma, \tau_h)_T + (e_u, \diver \tau_h)_T + (\diver e_\sigma, v_h)_T - j \w (e_u, v_h)_T \\
	&+ (\Pi_F \sigma \cdot \n - \Pi \sigma \cdot \n, v_h)_\elbnd
	- (e_{\hat u}, \tau_h \cdot \n)_\elbnd - (\Pi_F \sigma \cdot \n - \sigma_h \cdot \n, \hat v_h)_\elbnd \\
	&- \alpha (\Pi_F [u] - [u_h], [v_h])_\elbnd + \beta (\Pi_F \llbracket \sigma \rrbracket - \llbracket \sigma_h \rrbracket, \llbracket \tau_h \rrbracket)_\elbnd - (e_{\hat u}, \hat v_h)_\bnd.
\end{split}
\end{align}
There hold the following identities on element boundaries:
\begin{align}
\begin{split}
	\Pi_F \sigma \cdot \n - \sigma_h \cdot \n &= \Pi_F \sigma \cdot \n - \Pi \sigma \cdot \n + \Pi \sigma \cdot \n - \sigma_h \cdot \n \\
	&= e_\sigma \cdot \n - (\Pi \sigma \cdot \n - \Pi_F \sigma \cdot \n),
\end{split} \\
\begin{split}
	\Pi_F u - \Pi_F u - u_h + \hat u_h &= \Pi u - u_h - \Pi_F u + \hat u_h - \Pi u + \Pi_F u \\
	&= e_u - e_{\hat u} - (\Pi u - \Pi_F u),
\end{split} \\
\begin{split}
	\Pi_F \sigma \cdot \n - \Pi_F \sigma \cdot \n - \sigma_h \cdot \n + \hat \sigma_{\n, h} &= \Pi \sigma \cdot \n - \sigma_h \cdot \n - \Pi_F \sigma \cdot \n + \hat \sigma_{\n, h} - (\Pi \sigma \cdot \n - \Pi_F \sigma \cdot \n) \\
	&= e_\sigma \cdot \n - e_{\hat \sigma} - (\Pi \sigma \cdot \n - \Pi_F \sigma \cdot \n).
\end{split}
\end{align}
Replacing the appropriate terms results in
\begin{align}
\begin{split}
	0 =\ & \sum_{T \in \mathcal{T}} j \w (e_\sigma, \tau_h)_T + (e_u, \diver \tau_h)_T + (\diver e_\sigma, v_h)_T - j \w (e_u, v_h)_T \\
	&+ (\Pi_F \sigma \cdot \n - \Pi \sigma \cdot \n, v_h)_\elbnd
	- (e_{\hat u}, \tau_h \cdot \n)_\elbnd - (e_\sigma \cdot \n - (\Pi \sigma \cdot \n - \Pi_F \sigma \cdot \n), \hat v_h)_\elbnd \\
	&- \alpha (e_u - e_{\hat u} - (\Pi u - \Pi_F u), [v_h])_\elbnd + \beta (e_\sigma \cdot \n - e_{\hat \sigma} - (\Pi \sigma \cdot \n - \Pi_F \sigma \cdot \n), \llbracket \tau_h \rrbracket)_\elbnd - (e_{\hat u}, \hat v_h)_\bnd \\
	=\ & B(e_\sigma, e_{\hat \sigma}, e_u, e_{\hat u}, \tau_h, \hat \tau_{\n,h}, v_h, \hat v_h) \\
	&+ \sum_{T \in \mathcal{T}} - (\Pi \sigma \cdot \n - \Pi_F \sigma \cdot \n, [v_h])_\elbnd + \alpha (\Pi u - \Pi_F u, [v_h])_\elbnd - \beta (\Pi \sigma \cdot \n - \Pi_F \sigma \cdot \n, \llbracket \tau_h \rrbracket)_\elbnd,
\end{split}
\end{align}
which concludes the proof.
\end{proof}

With the previous lemma Theorem \ref{th::faceterror} can be shown.

\begin{proof}[Proof of Theorem \ref{th::faceterror}]
The proof is similar to \cite[Lemma 4.2]{maxwellabsolute}.
Setting $\tau_h := e_\sigma, \hat \tau_{\n, h} := e_{\hat \sigma}, v_h := - e_u, \hat v_h := - e_{\hat u}$ in Lemma \ref{lem::errorweak} yields
\begin{align}
\begin{split}
	&B(e_\sigma, e_{\hat \sigma}, e_u, e_{\hat u}, e_\sigma, e_{\hat \sigma}, -e_u, -e_{\hat u}) \\
	&=\sum_{T \in \mathcal{T}} j \w \lVert e_\sigma \rVert_T^2 + 2 \Im (e_u, \diver e_\sigma)_T + j \w \lVert e_u \rVert_T^2 - 2 \Im (e_{\hat u}, e_\sigma \cdot \n)_\elbnd
	+ \alpha \lVert [e_u] \rVert_\elbnd^2 + \beta \Vert \llbracket e_\sigma \rrbracket \rVert_\elbnd^2 + \lVert e_{\hat u} \rVert_\bnd^2 \\
	&= \sum_{T \in \mathcal{T}} - (\Pi \sigma \cdot \n - \sigma \cdot \n, [e_u])_\elbnd + \alpha (\Pi u - u, [e_u])_\elbnd + \beta (\Pi \sigma \cdot \n - \sigma \cdot \n, \llbracket e_\sigma \rrbracket)_\elbnd
\end{split}
\end{align}
and considering the real part gives
\begin{align}
\begin{split}
	\sum_{T \in \mathcal{T}} \alpha \lVert [e_u] \rVert_\elbnd^2 &+ \beta \Vert \llbracket e_\sigma \rrbracket \rVert_\elbnd^2 + \lVert e_{\hat u} \rVert_\bnd^2 \\
	&= \Re \sum_{T \in \mathcal{T}} - (\Pi \sigma \cdot \n - \sigma \cdot \n, [e_u])_\elbnd + \alpha (\Pi u - u, [e_u])_\elbnd + \beta (\Pi \sigma \cdot \n - \sigma \cdot \n, \llbracket e_\sigma \rrbracket)_\elbnd.
\end{split}
\end{align}
The right-hand side can be estimated by
\begin{align}
\begin{split}
	&\Bigg \vert \sum_{T \in \mathcal{T}} - (\Pi \sigma \cdot \n - \sigma \cdot \n, [e_u])_\elbnd + \alpha (\Pi u - u, [e_u])_\elbnd + \beta (\Pi \sigma \cdot \n - \sigma \cdot \n, \llbracket e_\sigma \rrbracket)_\elbnd \Bigg \vert \\
	&\leq \Bigg ( \sum_{T \in \mathcal{T}} \frac{2}{\alpha} \lVert \Pi \sigma \cdot \n - \sigma \cdot \n \rVert_\elbnd^2 + 2 \alpha \lVert \Pi u - u \rVert_\elbnd^2 + \beta \lVert \Pi \sigma \cdot \n - \sigma \cdot \n \rVert_\elbnd^2 \Bigg )^{\frac{1}{2}}
	\Bigg ( \sum_{T \in \mathcal{T}} \alpha \lVert [e_u] \rVert_\elbnd^2 + \beta \lVert \llbracket e_\sigma \rrbracket \rVert_\elbnd^2 \Bigg )^{\frac{1}{2}}
\end{split}
\end{align}
and due to the projection properties, there holds
\begin{align}
\begin{split}
	\Bigg ( \sum_{T \in \mathcal{T}} \frac{2}{\alpha} \lVert \Pi \sigma \cdot \n - \sigma \cdot \n \rVert_\elbnd^2 + 2 \alpha \lVert \Pi u - u \rVert_\elbnd^2 + \beta \lVert \Pi \sigma \cdot \n - \sigma \cdot \n \rVert_\elbnd^2 \Bigg )^{\frac{1}{2}} \\
	\leq C \Bigg ( \left (\frac{2}{\alpha} + \beta \right ) h^{2s + 1} \lVert \sigma \rVert_{s+1}^2 + 2 \alpha h^{2t + 1} \lVert u \rVert_{t+1}^2 \Bigg )^{\frac{1}{2}},
\end{split}
\end{align}
which concludes the proof.
\end{proof}

The previous result states optimal convergence rates for jumps.
In the following step Theorem \ref{th::scalarerror} is proven, similarly to \cite[Lemma 4.3]{maxwellabsolute}, by applying an Aubin-Nitsche trick with the right-hand side $f = e_u$.

\begin{proof}[Proof of Theorem \ref{th::scalarerror}]
Using an Aubin-Nitsche technique for the adjoint problem and inserting projections $\Pi$ and $\Pi_F$, as well as applying Lemma \ref{lem::adjprojectproperty} and Lemma \ref{lem::errorweak}, yields
\begin{align}
\begin{split}
	\lVert e_u \rVert^2_\domain 
	= &B(e_\sigma, e_{\hat \sigma}, e_u, e_{\hat u}, \phi, \phi_\n, w, w)\\
	= &B(e_\sigma, e_{\hat \sigma}, e_u, e_{\hat u}, \phi - \Pi \phi, \phi_\n - \Pi_F \phi_\n, w - \Pi w, w - \Pi_F w) \\
		&+ B(e_\sigma, e_{\hat \sigma}, e_u, e_{\hat u}, \Pi \phi, \Pi_F \phi_\n, \Pi w, \Pi_F w) \\
	=& \sum_{T \in \mathcal{T}} ([e_u], (\phi - \Pi \phi) \cdot \n)_\elbnd - \alpha ([e_u], w - \Pi w)_\elbnd + \beta (\llbracket e_\sigma \rrbracket, (\phi - \Pi \phi) \cdot \n)_\elbnd \\
		&+ \sum_{T \in \mathcal{T}} (\Pi \sigma \cdot \n - \sigma \cdot \n, [\Pi w])_\elbnd - \alpha (\Pi u - u, [\Pi w])_\elbnd + \beta (\Pi \sigma \cdot \n - \sigma \cdot \n, \llbracket \Pi \phi \rrbracket)_\elbnd.\end{split}
\end{align}
Estimating the terms leads to
\begin{align}
\begin{split}
	\lVert &e_u \rVert^2_\domain \\
	\leq &\Bigg ( \sum_{T \in \mathcal{T}} \alpha \lVert [e_u] \rVert_\elbnd^2 + \beta \lVert \llbracket e_\sigma \rrbracket \rVert_\elbnd^2 \Bigg )^{\frac{1}{2}}
	\Bigg ( \sum_{T \in \mathcal{T}} \frac{2}{\alpha} \lVert (\phi - \Pi \phi) \cdot \n \rVert_\elbnd^2
		+ 2\alpha \lVert w - \Pi w \rVert_\elbnd^2 + \beta \lVert (\phi - \Pi \phi)\cdot \n \rVert_\elbnd^2 \Bigg )^{\frac{1}{2}} \\
	+& \Bigg ( \sum_{T \in \mathcal{T}} \frac{2}{\alpha} \lVert (\sigma - \Pi \sigma) \cdot \n \rVert_\elbnd^2
		+ 2\alpha \lVert u - \Pi u \rVert_\elbnd^2 + \beta \lVert (\sigma - \Pi \sigma)\cdot \n \rVert_\elbnd^2 \Bigg )^{\frac{1}{2}}
	\Bigg ( \sum_{T \in \mathcal{T}} \alpha \lVert [\Pi w] \rVert_\elbnd^2 + \beta \lVert \llbracket \Pi \phi \rrbracket \rVert_\elbnd^2 \Bigg )^{\frac{1}{2}} \\
	\leq & C \Bigg ( \left (\frac{2}{\alpha} + \beta \right ) h^{2s + 1} \lVert \sigma \rVert_{s+1}^2 + 2 \alpha h^{2t + 1} \lVert u \rVert_{t+1}^2 \Bigg )^{\frac{1}{2}}
	\Bigg ( \left (\frac{2}{\alpha} + \beta \right ) h^{2p+1} \lVert \phi \rVert_{p+1}^2 + 2 \alpha h^{2 q + 1} \lVert w \rVert_{q+1}^2 \Bigg )^{\frac{1}{2}}.
\end{split}
\end{align}
The $H^2$ regularity of the adjoint problem implies $p=0, q = 1$ giving
\begin{align}
	\lVert e_u \rVert_\domain
	\leq C \left ( \left (\frac{2}{\alpha} + \beta \right ) \left ( 1 + \w \right )^2 h + 2 \alpha (1 + \w)^2 \w^2 h^{3} \right)^{\frac{1}{2}}
		\left ( \left (\frac{2}{\alpha} + \beta \right )h^{2s + 1} \lVert \sigma \rVert_{s+1}^2 + 2 \alpha h^{2t + 1} \lVert u \rVert_{t+1}^2 \right )^{\frac{1}{2}},
\end{align}
which concludes the proof.
\end{proof}
The applied Aubin-Nitsche technique is commonly used to prove the superconvergence of the projected error $e_u$.
This is not the case for the presented HDG formulation, because the additional rate, generated due to the reasoning of the adjoint regularity, 
compensates for the otherwise sub-optimal rate through estimating boundary by volume terms.

\section{Error Estimate for the Flux} \label{sec::sigmabestapprox}

To prove optimal convergence rates for the flux the issue due to $h$-independent stabilisation needs to be overcome.
In \cite{cockburn2010projection} a method based upon a specially devised projection is established and in \cite{griesmaier2011error}
this method is applied to an HDG formulation of the Helmholtz equation to asymptotically show optimal rates. 
Removing the second facet variable $\hat \sigma_\n$ and the $\beta$-stabilisation in \eqref{eq::weakform} would lead to the formulation therein.
Therefore, similar steps as in \cite{griesmaier2011error} lead to the desired result, but they differ due to the $\beta$-stabilisation. 
For $\beta = 0$ the projection introduced in this work falls back to the projection in \cite{cockburn2010projection}.

To establish the needed projection it is favourable to eliminate the facet variables $\hat \sigma_{\n}$ and $\hat u$ which results in an equivalent DG formulation.
To accommodate the elimination of facet variables DG-jumps will be defined.
\begin{lemma} \label{lem::facetellim}
The discrete facet variables $\hat \sigma_{\n,h}$ and $\hat u_h$ of the solution of \eqref{eq::weakform} are uniquely defined by
\begin{subequations}
\begin{align}
	\hat \sigma_{\n_{\pm}, h} &= \{ \sigma_h \}_{\n_\pm} := \frac{1}{2} (\sigma_{h,+} + \sigma_{h,-}) \cdot \n_\pm && \text{ on } F_{\rm I}, \label{eq::sigmafacet} \\
	\hat \sigma_{\n,h} &= \sigma_h \cdot \n, && \text{ on } F_{\rm O} \label{eq::sigmafacetbnd}
\end{align}
\end{subequations}
and
\begin{subequations}
\begin{align}
	\hat u_h &= \{ u_h \} - \frac{1}{2 \alpha} [\sigma_h]_\n := \frac{1}{2}(u_{h,+} + u_{h,-}) - \frac{1}{2 \alpha} (\sigma_{h,+} \cdot \n_+ + \sigma_{h,-} \cdot \n_-) && \text{ on } F_{\rm I}, \label{eq::ufacet} \\
	\hat u_h &= \frac{1}{1 + \alpha} ( \Pi_F g - \sigma_h \cdot \n + \alpha u_h) && \text{ on } F_{\rm O}. \label{eq::ufacetbnd}
\end{align}
\end{subequations}
The same holds for the continuous facet variables $\hat \sigma_{\n_\pm}$ and $\hat u$.
\end{lemma}
\begin{proof}
In the proof, the $h$-subscript for discrete variables is neglected.
Combining $\hat \tau_{\n}$ terms on an inner facet $F_{\rm I}$ leads to 
\begin{align}
	- \beta(\llbracket \sigma \rrbracket_+, \hat \tau_{\n_+})_F - \beta(\llbracket \sigma \rrbracket_-, \hat \tau_{\n_-})_F =
		&- \beta (\sigma_+ \cdot \n_+ + \sigma_- \cdot \n_+ - 2 \hat \sigma_{\n_+}, \hat \tau_{\n_+})_F,
\end{align}
which directly implies \eqref{eq::sigmafacet} and on outer facets $F_{\rm O}$ there holds
\begin{align}
	- \beta (\llbracket \sigma \rrbracket, \hat \tau_\n)_\bnd = 0
\end{align}
giving \eqref{eq::sigmafacetbnd}.
Similarly, for the scalar variable holds on inner facets
\begin{align}
	- ((\sigma_+ - \sigma_-) \cdot \n_+ - \alpha (u_+ + u_- - 2 \hat u), \hat v)_F = 0
\end{align}
yielding \eqref{eq::ufacet}.
Finally, the scalar variable on the domain boundary satisfies
\begin{align}
	-(\sigma \cdot \n, \hat v)_\bnd + \alpha (u - \hat u, \hat v)_\bnd - (\hat u, \hat v)_\bnd = - (g, \hat v)_\bnd
\end{align}
implying
\begin{align}
	(1+\alpha) \hat u = \Pi_F g - \sigma \cdot \n + \alpha u.
\end{align}
\end{proof}

With the previous lemma the facet variables can be eliminated from the HDG formulation and the volume variables satisfy the following DG formulation.

\begin{lemma}[DG Formulation]
For given $g \in \ltwo(\bnd)$ and $\w > 0$ find $\sigma \in \hdiv$, $u \in \ltwo(\domain)$, such that
\begin{align}
\begin{split}
	B_{DG} (\sigma, u, \tau, v) = \frac{1}{1+\alpha} (g, \tau \cdot \n - \alpha v)_\bnd
\end{split}
\end{align}
holds for all $\tau \in \hdiv$, $v \in \ltwo(\domain)$
with $\alpha := 1$, $\beta := 1$ and the SLF
\begin{align}
\begin{split}
	B_{DG} (\sigma, u, \tau, v) := &\sum_{T \in \mathcal{T}} j \w (\sigma, \tau)_T + (u, \diver \tau)_T + (\diver \sigma, v)_T - j \w (u, v)_T \\
	 + & \sum_{F_{\rm I} \in \mathcal{F}} - (\{u\}, [\tau]_\n )_{F_{\rm I}}
	- ([\sigma]_\n, \{v\})_{F_{\rm I}} - \frac{\alpha}{2} ([u]_+, [v]_+)_{F_{\rm I}}
	+ \left ( \frac{1}{2 \alpha} + \frac{\beta}{2} \right ) ([\sigma]_\n, [\tau]_\n)_{F_{\rm I}} \\
	+ & \ \frac{1}{1+\alpha} \left ( (\sigma \cdot \n, \tau \cdot \n)_\bnd - \alpha \left ( (u, \tau \cdot \n)_\bnd + (\sigma \cdot \n, v)_\bnd + (u,v)_\bnd \right ) \right ).
\end{split}
\end{align}
The scalar jump on inner facets is defined as $[u]_+ := u_+ - u_-$.
\end{lemma}
\begin{proof}
With Lemma \ref{lem::facetellim} the facet variables on inner facets are eliminated. Terms on adjacent elements are combined into DG form:
\begin{align}
	\beta (\llbracket \sigma \rrbracket_+, \tau_+ \cdot \n_+)_{F_{\rm I}} + \beta (\llbracket \sigma \rrbracket_-, \tau_- \cdot \n_-)_{F_{\rm I}} &= \frac{\beta}{2} ( [\sigma]_\n, [\tau]_\n)_{F_{\rm I}}, \\
	- (\hat u, [\tau]_\n)_{F_{\rm I}} &= - ( \{ u \} - \frac{1}{2 \alpha} [\sigma]_\n, [\tau]_\n )_{F_{\rm I}}, \\
	- \alpha ([u]_+, v_+)_F - \alpha ([u]_-, v_-)_{F_{\rm I}} &= - ( [\sigma]_\n, \{ v \} )_{F_{\rm I}} - \frac{\alpha}{2} ([u]_+, [v]_+)_{F_{\rm I}}.
\end{align}
Replacing $\hat u$ on the domain boundary leads to
\begin{align}
\begin{split}
	-(\hat u, \tau \cdot \n)_\bnd - \alpha (u - \hat u, v)_\bnd
	= &- \left (\frac{1}{1 + \alpha} ( g - \sigma \cdot \n + \alpha u), \tau \cdot \n \right )_\bnd \\
		&- \alpha \left (u - \frac{1}{1 + \alpha} ( g - \sigma \cdot \n + \alpha u), v \right)_\bnd \\
	= & \frac{1}{1 + \alpha} (\sigma \cdot \n, \tau \cdot \n)_\bnd - \frac{\alpha}{1 + \alpha} (u, \tau \cdot \n)_\bnd - \frac{1}{1 + \alpha} (g, \tau \cdot \n)_\bnd \\
		&- \alpha (u,v)_\bnd - \frac{\alpha}{1+\alpha} (\sigma \cdot \n, v)_\bnd + \frac{\alpha^2}{1+\alpha} (u, v)_\bnd + \frac{\alpha}{1+\alpha}(g, v)_\bnd \\
	= & \frac{1}{1+\alpha} \left ( (\sigma \cdot \n, \tau \cdot \n)_\bnd - \alpha \left ( (u, \tau \cdot \n)_\bnd + (\sigma \cdot \n, v)_\bnd + (u,v)_\bnd \right ) \right ) \\
		&- \frac{1}{1+\alpha} (g, \tau \cdot \n - \alpha v)_\bnd.
\end{split}
\end{align}
\end{proof}

The next step towards an error bound for the flux is to introduce a suitable projection.
The first result holds for general projections inserted into the DG formulation.
\begin{lemma} \label{lem::dgproject}
Let $P$ be a projection
\begin{align}
	P : (\sigma, u) \mapsto (P_\sigma (\sigma, u), P_u (\sigma, u)) \in \mathcal{RT}^p(\mathcal{T}) \times \mathcal{P}^p(\mathcal{T})
\end{align}
into discrete spaces.
With projected errors defined as
\begin{align}
	e_\sigma^P := P_\sigma (\sigma, u) - \sigma_h, && e_u^P := P_u (\sigma, u) - u_h,
\end{align}
there holds
\begin{align}
\begin{split}
	 \w \sum_{T \in \mathcal{T}} \lVert e_\sigma^P \rVert_T^2 = - \Im \Bigg ( \sum_{T \in \mathcal{T}} &  j \w (\sigma - P_\sigma(\sigma, u), e_\sigma^P)_T + (u - P_u (\sigma, u), \diver e_\sigma^P)_T \\
	& - (\sigma - P_\sigma(\sigma, u), \grad e_u^P)_T - j \w (u - P_u (\sigma, u), e_u^P)_T - j \w \lVert e_u^P \rVert^2_T \Bigg )\\
	- \Im \Bigg ( \sum_{F_{\rm I} \in \mathcal{F}} & \left (\{ \sigma - P_\sigma (\sigma, u) \}_{\n_+} - \frac{\alpha}{2} [u - P_u (\sigma, u)]_+, [e_u^P]_+ \right)_{F_{\rm I}} \\
	&+ \left( \left ( \frac{1}{2 \alpha} + \frac{\beta}{2} \right ) [\sigma - P_\sigma (\sigma, u)]_\n - \{u - P_u (\sigma, u) \}, [e_\sigma^P]_\n \right)_{F_{\rm I}} \Bigg ) \\
	- \frac{1}{1+\alpha} \Im & \left ((\sigma - P_\sigma)\cdot \n - \alpha (u - P_u), e_\sigma^P \cdot \n + e_u^P \right )_\bnd.
\end{split}
\end{align}
\end{lemma}
\begin{proof}
The main idea is a repetition of the arguments in the proof of Lemma \ref{lem::adjprojectproperty}.
To shorten the notation the arguments $(\sigma, u)$ of $P_\sigma(\sigma, u)$ and $P_u (\sigma, u)$ will be omitted in the proof.
Applying the Galerkin orthogonality and inserting the projection yields
\begin{align}
\begin{split}
	0 =& B_{DG}(\sigma - \sigma_h, u - u_h, u - \hat u_h, e_\sigma^P, e_u^P, e_{\hat u}) \\
	=& \sum_{T \in \mathcal{T}} j \w (\sigma - P_\sigma + e_\sigma^P, e_\sigma^P)_T + (u - P_u + e_u^P, \diver e_\sigma^P)_T \\
	& + (\diver (\sigma - P_\sigma + e_\sigma^P), e_u^P)_T - j \w (u - P_u + e_u^P, e_u^P)_T \\
	+ & \sum_{F_{\rm I} \in \mathcal{F}} - (\{u - P_u + e_u^P\}, [e_\sigma^P]_\n )_{F_{\rm I}}
	- ([\sigma - P_\sigma + e_\sigma^P]_\n, \{e_u^P\})_{F_{\rm I}} \\
	&- \frac{\alpha}{2} ([u - P_u + e_u^P]_+, [e_u^P]_+)_{F_{\rm I}}
	+ \left ( \frac{1}{2 \alpha} + \frac{\beta}{2} \right ) ([\sigma - P_\sigma + e_\sigma^P]_\n, [e_\sigma^P]_\n)_{F_{\rm I}} \\
	+& \frac{1}{1+\alpha} \Big ( ((\sigma - P_\sigma + e_\sigma^P) \cdot \n, e_\sigma^P \cdot \n)_\bnd - \alpha (u - P_u + e_u^P, e_\sigma^P \cdot \n)_\bnd  \\
	&- \alpha ((\sigma - P_\sigma + e_\sigma^P) \cdot \n, e_u^P)_\bnd - \alpha (u - P_u + e_u^P, e_u^P)_\bnd \Big ).
\end{split} \label{eq::BDGanal}
\end{align}
For the volume terms there holds after partial integration
\begin{align}
\begin{split}
	\sum_{T \in \mathcal{T}}& j \w (\sigma - P_\sigma + e_\sigma^P, e_\sigma^P)_T + (u - P_u + e_u^P, \diver e_\sigma^P)_T \\
	& + (\diver (\sigma - P_\sigma + e_\sigma^P), e_u^P)_T - j \w (u - P_u + e_u^P, e_u^P)_T \\
	= &\sum_{T \in \mathcal{T}} j \w (\sigma - P_\sigma, e_\sigma^P)_T + j \w \lVert e_\sigma^P \rVert_T^2 + (u - P_u, \diver e_\sigma^P)_T + 2 \Re (e_u^P, \diver e_\sigma^P)_T \\
	& - (\sigma - P_\sigma, \grad e_u^P)_T + ((\sigma - P_\sigma) \cdot \n, e_u^P)_\elbnd - j \w (u - P_u, e_u^P)_T - j \w \lVert e_u^P \rVert^2_T
\end{split}
\end{align}
and connecting the arising boundary terms on inner facets leads to
\begin{align}
	((\sigma - P_{\sigma,+})  \cdot \n_+, e_{u,+}^{P})_{F_{\rm I}} + ((\sigma - P_{\sigma,-})  \cdot \n_-, e_{u,-}^{P})_{F_{\rm I}} = (\{ \sigma - P_\sigma \}_{\n_+}, [e_u^P]_+)_{F_{\rm I}} + ([\sigma - P_\sigma]_\n, \{e_u^P\})_{F_{\rm I}}.
\end{align}
Combining these with the facet terms in \eqref{eq::BDGanal} results in
\begin{align}
\begin{split}
	\sum_ {F_{\rm I} \in \mathcal{F}} &  - (\{u - P_u + e_u^P\}, [e_\sigma^P]_\n )_{F_{\rm I}}
	- ([\sigma - P_\sigma + e_\sigma^P]_\n, \{e_u^P\})_{F_{\rm I}} \\
	&- \frac{\alpha}{2} ([u - P_u + e_u^P]_+, [e_u^P]_+)_{F_{\rm I}}
	+ \left ( \frac{1}{2 \alpha} + \frac{\beta}{2} \right ) ([\sigma - P_\sigma + e_\sigma^P]_\n, [e_\sigma^P]_\n)_{F_{\rm I}}\\
	&+ (\{ \sigma - P_\sigma \}_{\n_+}, [e_u^P]_+)_{F_{\rm I}} + ([\sigma - P_\sigma]_\n, \{e_u^P\})_{F_{\rm I}} \\
	= \sum_{F_{\rm I} \in \mathcal{F}} & - 2 \Re (\{e_u^P\},[e_\sigma^P]_\n)_{F_{\rm I}}
	+ \left (\{ \sigma - P_\sigma \}_{\n_+} - \frac{\alpha}{2} [u - P_u]_+, [e_u^P]_+ \right)_{F_{\rm I}} - \frac{\alpha}{2} \lVert [e_u^P]_+ \rVert^2_{F_{\rm I}} \\
	&+ \left( \left ( \frac{1}{2 \alpha} + \frac{\beta}{2} \right ) [\sigma - P_\sigma]_\n - \{u - P_u\}, [e_\sigma^P]_\n \right)_{F_{\rm I}} + \left ( \frac{1}{2 \alpha} + \frac{\beta}{2} \right ) \lVert [e_\sigma^P]_\n \rVert^2_{F_{\rm I}}.
\end{split}
\end{align}
Finally, on the domain boundary there holds, considering the additional term arising from partial integration,
\begin{align}
\begin{split}
	& \frac{1}{1+\alpha} \Big ( ((\sigma - P_\sigma + e_\sigma^P) \cdot \n, e_\sigma^P \cdot \n)_\bnd - \alpha (u - P_u + e_u^P, e_\sigma^P \cdot \n)_\bnd  \\
	&- \alpha ((\sigma - P_\sigma + e_\sigma^P) \cdot \n, e_u^P)_\bnd - \alpha (u - P_u + e_u^P, e_u^P)_\bnd \Big ) + ((\sigma - P_\sigma) \cdot \n, e_u^P)_\bnd \\
	=& \frac{1}{1+\alpha} \Big ( ((\sigma - P_\sigma) \cdot \n, e_\sigma^P \cdot \n)_\bnd + \lVert e_\sigma^P \cdot \n \rVert^2_\bnd 
		- \alpha (u - P_u, e_\sigma^P \cdot \n)_\bnd - 2 \alpha \Re (e_u^P, e_\sigma^P \cdot \n)_\bnd \\
	&+ ((\sigma - P_\sigma) \cdot \n, e_u^P)_\bnd
		- \alpha (u - P_u, e_u^P)_\bnd - \alpha \lVert e_u^P \rVert^2_\bnd \Big ) \\
	=& \frac{1}{1+\alpha} \Big ( ((\sigma - P_\sigma) \cdot \n - \alpha (u - P_u), e_\sigma^P \cdot \n)_\bnd + \lVert e_\sigma^P \cdot \n \rVert^2_\bnd 
		- 2 \alpha \Re (e_u^P, e_\sigma^P \cdot \n)_\bnd \\
	&+ ((\sigma - P_\sigma) \cdot \n - \alpha (u - P_u), e_u^P)_\bnd
		- \alpha \lVert e_u^P \rVert^2_\bnd \Big ) \\
	=& \frac{1}{1+\alpha} \Big ( ((\sigma - P_\sigma) \cdot \n - \alpha (u - P_u), e_\sigma^P \cdot \n + e_u^P)_\bnd + \lVert e_\sigma^P \cdot \n \rVert^2_\bnd 
		- 2 \alpha \Re (e_u^P, e_\sigma^P \cdot \n)_\bnd
		- \alpha \lVert e_u^P \rVert^2_\bnd \Big ).
\end{split}
\end{align}
Taking the imaginary part implies
\begin{align}
	0 =  \Im  B_{DG} &(\sigma - \sigma_h, u - u_h, e_{\hat u}, e_\sigma^P, e_u^P, e_{\hat u}) \\
		= \Im \sum_{T \in \mathcal{T}} &  j \w (\sigma - P_\sigma, e_\sigma^P)_T + j \w \lVert e_\sigma^P \rVert_T^2 + (u - P_u, \diver e_\sigma^P)_T \\
	& - (\sigma - P_\sigma, \grad e_u^P)_T - j \w (u - P_u, e_u^P)_T - j \w \lVert e_u^P \rVert^2_T \\
	+ \Im \sum_{F_{\rm I} \in \mathcal{F}} & \left (\{ \sigma - P_\sigma \}_{\n_+} - \frac{\alpha}{2} [u - P_u]_+, [e_u^P]_+ \right)_{F_{\rm I}} \\
	&+ \left( \left ( \frac{1}{2 \alpha} + \frac{\beta}{2} \right ) [\sigma - P_\sigma]_\n - \{u - P_u\}, [e_\sigma^P]_\n \right)_{F_{\rm I}} \\
	+ \frac{1}{1+\alpha} \Im & ((\sigma - P_\sigma)\cdot \n - \alpha (u - P_u), e_\sigma^P \cdot \n + e_u^P)_\bnd,
\end{align}
which concludes the proof.
\end{proof}

From the previous lemma, it can be seen that a suitable projection for vanishing facet terms needs to satisfy
\begin{align}
	\left ( \{ \sigma - P_\sigma \}_{\n_+} - \frac{\alpha}{2} [u - P_u]_+, \mu_h \right )_{F_{\rm I}} = 0 && \forall \mu_h \in \mathcal{P}^p(F_{\rm I}), \\
	\left ( \left ( \frac{1}{2 \alpha} + \frac{\beta}{2} \right ) [\sigma - P_\sigma]_\n - \{u - P_u\}, \nu_h \right )_{F_{\rm I}} = 0 && \forall \nu_h \in \mathcal{P}^p(F_{\rm I}),
\end{align}
on inner facets.
The two conditions per facet can be reformulated into the following equivalent form
\begin{align}
	\left ( \frac{1}{\alpha} + \beta \right ) ((\sigma - P_{\sigma,+}) \cdot \n_+, \mu_h)_{F_{\rm I}} &= (u - P_{u,+}, \mu_h)_{F_{\rm I}} + \frac{\alpha \beta}{2} ([u - P_u]_+), \mu_h)_{F_{\rm I}}, \\
	\left ( \frac{1}{\alpha} + \beta \right ) ((\sigma - P_{\sigma,-}) \cdot \n_-, \nu_h)_{F_{\rm I}} &= (u - P_{u,-}, \nu_h)_{F_{\rm I}} + \frac{\alpha \beta}{2} ([u - P_u]_-), \nu_h)_{F_{\rm I}}.
\end{align}
Without a $\beta$-stabilisation this would exactly be the projection in \cite{cockburn2010projection, griesmaier2011error, Sayas2013FromRT}.
Collecting all required properties of the projection looks as follows.

\begin{definition} \label{def::Projection}
Let $P$ be the projection
\begin{align}
	P : (\sigma, u) \mapsto (P_\sigma (\sigma, u), P_u (\sigma, u)) \in \mathcal{RT}^p(\mathcal{T}) \times \mathcal{P}^p(\mathcal{T})
\end{align}
defined by
\begin{subequations}
\begin{align}
	(P_\sigma (\sigma, u), \tau_h)_T &= (\sigma, \tau_h)_T &&\forall \tau_h \in [\mathcal{P}^{p-1}(T)]^d, \label{eq:sigmaproj}\\
	(P_u (\sigma, u), v_h)_T &= (u, v_h)_T &&\forall v_h \in \mathcal{P}^p(T), \label{eq:uproj}\\
	\left ( \frac{1}{2 \alpha} + \frac{\beta}{2} \right ) ([\sigma - P_\sigma]_+, \mu_h)_{F_{\rm I}} &= (\{u - P_u\}, \mu_h)_{F_{\rm I}}  && \forall \mu_h \in \mathcal{P}^{p}({F_{\rm I}}), \label{eq:sigmaelbnd} \\
	(\{ \sigma - P_\sigma \}_+, \mu_h)_{F_{\rm I}} &= \frac{\alpha}{2} ([u - P_u]_+, \mu_h )_{F_{\rm I}} && \forall \mu_h \in \mathcal{P}^{p}({F_{\rm I}}), \\
	(\sigma \cdot \n - P_\sigma \cdot \n, \mu_h)_{F_{\rm O}} &= \alpha ( u - P_u, \mu_h)_{F_{\rm O}} &&\forall \mu_h \in \mathcal{P}^{p}({F_{\rm O}}). \label{eq:sigmabnd}
\end{align}
\end{subequations}
\end{definition}
On the domain boundary, the projection has to have the same properties as the projection in  \cite{cockburn2010projection}.
For $P$ existence and uniqueness, as well as the projection property and a suitable approximation property need to be proven.
The definition represents a square linear system of equations, therefore uniqueness of the projection automatically implies its existence.

The projection $P_u$ is decoupled from $\sigma$.
The properties for $P_u$ are proven in the following lemma.
\begin{lemma}[Projection Decoupling]
The projection $P_u$ is uniquely defined by
\begin{align}
	(P_u(\sigma, u), v_h)_T = (u, v_h)_T &&\forall v_h \in \mathcal{P}^{p}(T)
\end{align}
and only depends on $u$, therefore $P_u(\sigma, u) = P_u(u)$.
For $u \in H^{p+1}(T)$ there holds
\begin{align}
	\lVert u - P_u(u) \rVert_T \leq C h^{p+1} \vert u \vert_{H^{p+1}(T)},
\end{align}
with a constant $C > 0$ independent of $h, \w, \alpha, \beta$.
\end{lemma}
\begin{proof}
Equation \eqref{eq:uproj} implies that $P_u(\sigma, u)$ is the $\ltwo$-projection $\Pi u$.
\end{proof}

The proof for $P_\sigma$ is more involved and similar to the analysis in \cite{cockburn2010projection, Sayas2013FromRT}.

\begin{lemma} \label{lem::helper}
The space $\mathcal{P}^p_\perp(T)$ is defined by
\begin{align}
	\mathcal{P}^p_\perp (T) := \{u \in \mathcal{P}^p(T) : (u, v)_T = 0, \forall v \in \mathcal{P}^{p-1}(T) \}.
\end{align}
If $u \in P^p_\perp(T)$ satisfies $u = 0$ on a facet of $T$ then $u \equiv 0$ on the whole element.
\end{lemma}
\begin{proof}
See \cite[Lemma A.1]{cockburn2010projection} and \cite[Lemma 2.1]{Sayas2013FromRT}.
\end{proof}

\begin{lemma} \label{lem::helper2}
The space $\mathcal{RT}^p_\perp (T)$ is defined by
\begin{align}
	\mathcal{RT}^p_\perp (T) := \{\sigma \in \mathcal{RT}^p(T) : (\sigma, \tau)_T = 0, \forall \tau \in [\mathcal{P}^{p-1}(T)]^d \}.
\end{align}
Assume $\sigma \in \mathcal{RT}^p_\perp(T)$ then there holds
\begin{align}
	\lVert \sigma \rVert_T \leq C h^{\frac{1}{2}} \lVert \sigma \cdot \n \rVert_\elbnd,
\end{align}
with a constant $C > 0$ independent of $h, \w, \alpha, \beta$.
\end{lemma}
\begin{proof}
The proof is similar to \cite[Proposition A.3]{cockburn2010projection} and \cite[Lemma 2.1]{Sayas2013FromRT}.
First, it is shown that the boundary term is a norm on the space $\mathcal{RT}^p_\perp(T)$.
Assuming $\sigma \cdot \n = 0$ on $\elbnd$ then there holds
\begin{align}
	\lVert \diver \sigma \rVert_T^2 = (\sigma \cdot \n, \diver \sigma)_\elbnd - (\sigma, \grad \diver \sigma)_T = 0,
\end{align}
because $\grad \diver \sigma \in [\mathcal{P}^{p-1}(T)]^d$.
According to \cite[Proposition 2.3]{Sayas2013FromRT} this implies $\sigma \in [\mathcal{P}^{p}(T)]^d$ and by splitting $\sigma$ into
\begin{align}
	\sigma = \sum_{i=1}^{d-1} \sigma \cdot \n_i
\end{align}
there holds $\sigma \cdot \n_i \in \mathcal{P}^{p}_\perp (T)$ as well as $\sigma \cdot \n_i = 0$ on the facet $F_{\rm I}$. 
Then Lemma \ref{lem::helper} implies that $\sigma \cdot \n_i$ vanishes on the whole element
and therefore $\sigma = 0$. The estimate is proven by a standard scaling argument.
\end{proof}

With this lemma, the uniqueness and approximation property of $P_\sigma$ can be proven.

\begin{lemma} \label{lem::psigmaapproximation}
Assuming $H^{p+1}$-regularity of $\sigma$ and $u$, there exists a constant $C > 0$, independent of $\w, h, \alpha, \beta$, so that
\begin{align}
	\lVert \sigma - P_\sigma (\sigma, u) \rVert_\domain \leq C h^{p+1} \left ( \vert \sigma \vert_{H^{p+1}} + \frac{1 + \alpha \beta}{\alpha^{-1} + \beta} \vert u \vert_{H^{p+1}} \right ).
\end{align}
\end{lemma}
\begin{proof}
The proof is an adaptation of the approach in \cite[Proposition A.3]{cockburn2010projection}.
Consider the standard $\mathcal{RT}$-interpolant satisfying
\begin{subequations}
\begin{align}
	(\mathcal{RT}(\sigma), \tau_h)_T &= (\sigma, \tau_h)_T &&\forall \tau_h \in [\mathcal{P}^{p-1}(T)]^d, \\
	((\sigma - \mathcal{RT}(\sigma)) \cdot \n, \mu_h)_F &= 0  && \forall \mu_h \in \mathcal{P}^{p}(F)
\end{align}
\end{subequations}
and define $\delta_\sigma := \mathcal{RT} (\sigma) - P_\sigma(\sigma, u)$.
Due to \eqref{eq:sigmaproj} and (\ref{eq:sigmaelbnd} - \ref{eq:sigmabnd}) there holds for $\delta_\sigma$
\begin{align}
	(\delta_\sigma, \tau_h)_T &= 0 &&\forall \tau_h \in [\mathcal{P}^{p-1}(T)]^d, \label{eq::projortho}\\
	\left ( \frac{1}{\alpha} + \beta \right ) (\delta_\sigma \cdot \n, \mu_h)_{\elbnd \cap F_{\rm I}} 
		&= (u - P_u, \mu_h)_{\elbnd \cap F_{\rm I}} + \frac{\alpha \beta}{2} ([u - P_u], \mu_h)_{\elbnd \cap F_{\rm I}} && \forall \mu_h \in \mathcal{P}^{p}(F_{\rm I}), \\
	\left ( \frac{1}{\alpha} + \beta \right ) (\delta_\sigma \cdot \n, \mu_h)_{\elbnd \cap F_{\rm O}}
		&= (1 + \alpha \beta) ( u - P_u, \mu_h)_{\elbnd \cap F_{\rm O}} &&\forall \mu_h \in \mathcal{P}^{p}(F_{\rm O}).
\end{align}
The idea is to choose $\mu_h = \delta_\sigma \cdot \n$ in the equations above and to estimate facet terms by
\begin{align}
\begin{split}
	\vert (u- P_u, \delta_\sigma \cdot \n)_F \vert &\leq C h^{-\frac{1}{2}} \lVert u - P_u \rVert_F \lVert \delta_\sigma \rVert_T \\
		&\leq C h^{-1} (\lVert u - P_u \rVert_T + h \vert u - P_u \vert_{H^{1}(T)}) \lVert \delta_\sigma \rVert_T \\
		&\leq C h^{p} \vert u \vert_{H^{p+1}} \lVert \delta_\sigma \rVert_T.
\end{split}
\end{align}
The constant $C$ changes in each line, but stays independent of $h, \w, \alpha, \beta$.
Due to \eqref{eq::projortho} Lemma \ref{lem::helper2} can be applied yielding
\begin{align}
\begin{split}
	\left ( \frac{1}{\alpha} + \beta \right ) \lVert \delta_\sigma \rVert^2_T
		\leq& \left ( \frac{1}{\alpha} + \beta \right ) C h \lVert \delta_\sigma \cdot \n \rVert^2_\elbnd \\
	=& \ C h ( (u - P_u, \delta_\sigma \cdot \n)_{\elbnd \setminus \bnd} + \frac{\alpha \beta}{2} ([u - P_u], \delta_\sigma \cdot \n)_{\elbnd \setminus \bnd} \\
		&+ (1 + \alpha \beta) (u - P_u, \delta_\sigma \cdot \n)_{\elbnd \cap \bnd} ) \\
	\leq& \ C (1 + \alpha \beta) h^{p+1} \vert u \vert_{H^{p+1}} \lVert \delta_\sigma \rVert_T,
\end{split}
\end{align}
which implies
\begin{align}
	\lVert \delta_\sigma \rVert_T \leq C \frac{1 + \alpha \beta}{\alpha^{-1} + \beta} h^{p+1} \vert u \vert_{H^{p+1}}
\end{align}
and finally gives
\begin{align}
\begin{split}
	\lVert \sigma - P_\sigma (\sigma, u) \rVert_\domain &\leq \lVert \sigma - \mathcal{RT}(\sigma) \rVert_\domain + \lVert \delta_\sigma \rVert_\domain \\
	&\leq C h^{p+1} \left ( \vert \sigma \vert_{H^{p+1}} + \frac{1 + \alpha \beta}{\alpha^{-1} + \beta} \vert u \vert_{H^{p+1}} \right ).
\end{split}
\end{align}
\end{proof}

With this result, the error estimate for the flux can be finalised.

\begin{proof}[Proof of Theorem \ref{th::sigmaerror}]
According to Lemma \ref{lem::dgproject} in combination with the projection in Definition \ref{def::Projection} there holds
\begin{align}
\begin{split}
	 \w \sum_{T \in \mathcal{T}} \lVert e_\sigma^P \rVert_T^2 =& - \Im \Big ( \sum_{T \in \mathcal{T}}  j \w (\sigma - P_\sigma(\sigma, u), e_\sigma^P)_T
		- j \w \lVert e_u^P \rVert^2_T \Big ) \\
	\leq&\ \w \sum_{T \in \mathcal{T}} \vert (\sigma - P_\sigma(\sigma, u), e_\sigma^P)_T \vert + \lVert e_u^P \rVert^2_T \\
	\leq&\ \w \sum_{T \in \mathcal{T}} \lVert \sigma - P_\sigma(\sigma, u) \rVert_T \lVert e_\sigma^P \rVert_T + \lVert e_u^P \rVert^2_T.
\end{split}
\end{align}
Applying Young's inequality, Lemma \ref{lem::psigmaapproximation} for $\sigma - P_\sigma$ and Theorem \ref{th::scalarerror} for $e_u^P$ yields
\begin{align}
\begin{split}
	 \w \sum_{T \in \mathcal{T}} \lVert e_\sigma^P \rVert_T^2 
	 	&\leq \w \sum_{T \in \mathcal{T}} \lVert \sigma - P_\sigma(\sigma, u) \rVert_T^2 +  2 \lVert e_u^P \rVert^2_T \\
		&\leq \w C h^{2p+2} \left ( \lVert \sigma \rVert_{H^{p+1}}^2 + \left ( \frac{1 + \alpha \beta}{\alpha^{-1} + \beta} \right )^2 \lVert u \rVert_{H^{p+1}}^2 \right ) \\
		&+ \w C h^{2p+2} \left ( \left (\frac{2}{\alpha} + \beta \right ) \left ( 1 + \w \right )^2 + 2 \alpha (1 + \w)^2 \w^2 h^{2} \right) 
			\left ( \left (\frac{2}{\alpha} + \beta \right ) \lVert \sigma \rVert_{H^{p+1}}^2 + 2 \alpha  \lVert u \rVert_{H^{p+1}}^2 \right ).
\end{split}
\end{align}
Considering
\begin{align}
	\lVert \sigma - \sigma_h \rVert_\domain \leq \lVert \sigma - P_\sigma (\sigma, u) \rVert_\domain + \lVert e_\sigma^P \rVert_\domain
\end{align}
and applying Lemma \ref{lem::psigmaapproximation}
concludes the proof.
\end{proof}

\section{Numerical Experiments} \label{sec::numres}

The mixed Helmholtz problem \eqref{eq::weakform} has been discretised with the open source FE software NETGEN/NGSolve \cite{jsnetgen, jsngsolve}.
It provides the required high-order FE spaces as well as the tools for the iterative solver.
First, the established error estimates are probed by simulations on convex domains with plane waves as analytical solutions.

\subsection{2D Plane Wave}

A 2D plane wave example has been considered on the square domain $\domain = (0,1)^2$ discretised by a structured mesh comprised of triangles.
For fixed $N \in \N_{+}$ the domain has been split into $N \times N$ identical squares and each smaller square was split into two triangles by a line from the top left to the bottom right corner.
In Figure \ref{fig:2dstructuredmesh} the mesh for $N=4$ is shown.
\begin{figure}
        	\centering
	\begin{subfigure}[b]{0.4\textwidth}
        		\centering
		\includegraphics[width=0.60\textwidth]{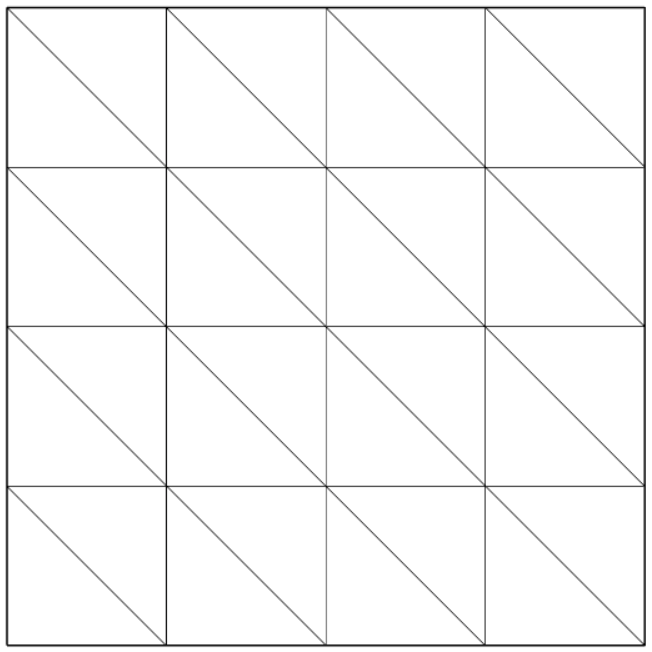}
       		\caption{The 2D structured mesh of the unit square is shown for the choice of $N=4$.}
      		\label{fig:2dstructuredmesh}
    	\end{subfigure}
	\hspace{1cm}
	\begin{subfigure}[b]{0.4\textwidth}
        		\centering
		\includegraphics[width=0.6\textwidth]{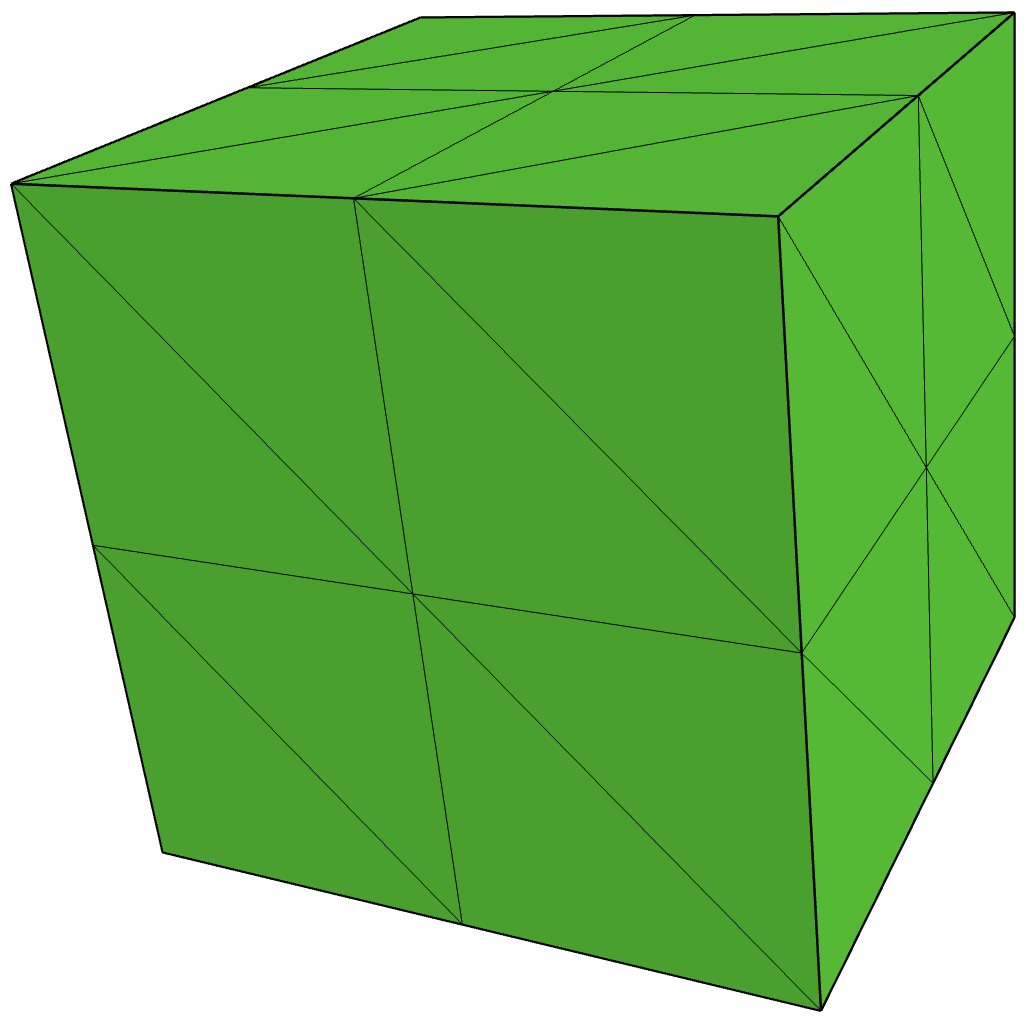}
       		\caption{The 3D structured mesh of the unit cube is shown for the choice of $N=2$.}
      		\label{fig:3dstructuredmesh}
    	\end{subfigure}
       	\caption{The structured meshes for the plane wave simulations in 2D and 3D are visualised.}
\end{figure}
The excitation $g$ in the HDG formulation \eqref{eq::weakform} has been chosen so that the analytical solution is the plane wave
\begin{align}
	u(x,y) = e^{j \kappa (x \cos \theta + y \sin \theta)}
\end{align}
with an angle of $\theta = \pi / 6$. The wave number was fixed at $\kappa = 5$ and structured meshes corresponding to $N = \{2,4,8,16,32,64,128\}$ have been considered.
In Figure \ref{fig:2dconv} the convergence rates for the polynomial degrees $p = \{0,1,2,3\}$ with respect to the mesh size are visualised.
\begin{figure}
	\centering
	\begin{subfigure}[b]{0.49\textwidth}
        		\centering
		\includegraphics[width=\textwidth]{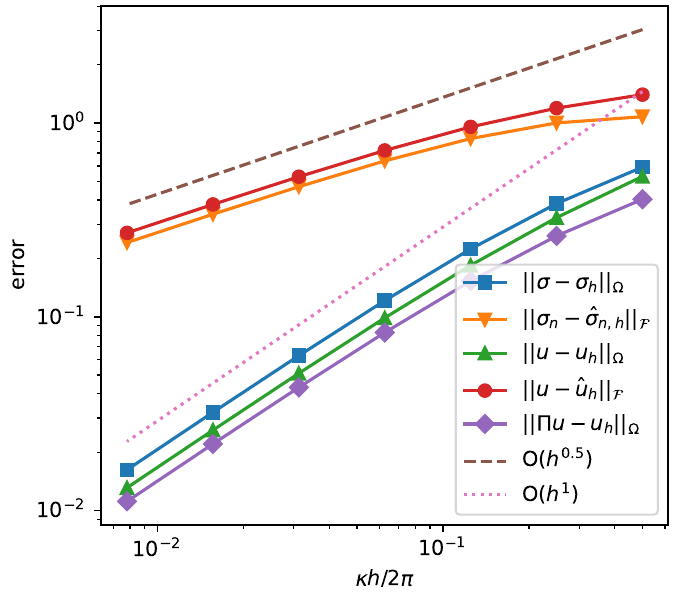}
       		\caption{$p=0$}
    	\end{subfigure}
	\begin{subfigure}[b]{0.49\textwidth}
        		\centering
		\includegraphics[width=\textwidth]{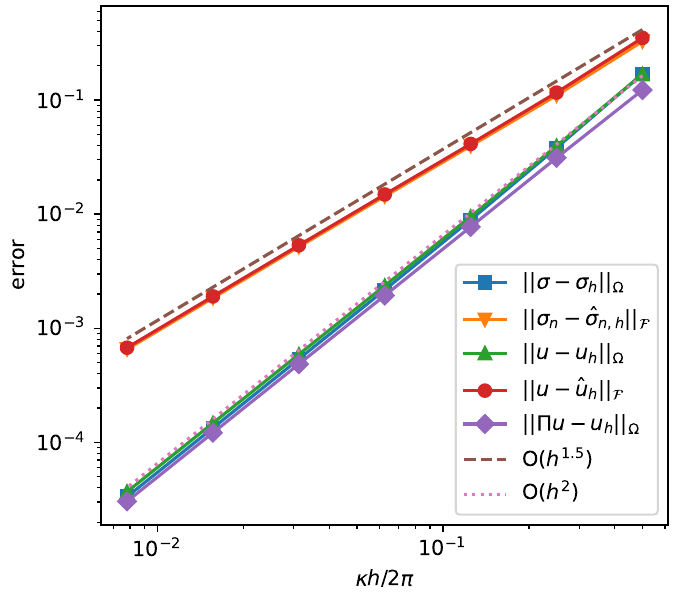}
       		\caption{$p=1$}
    	\end{subfigure}
	\begin{subfigure}[b]{0.49\textwidth}
        		\centering
		\includegraphics[width=\textwidth]{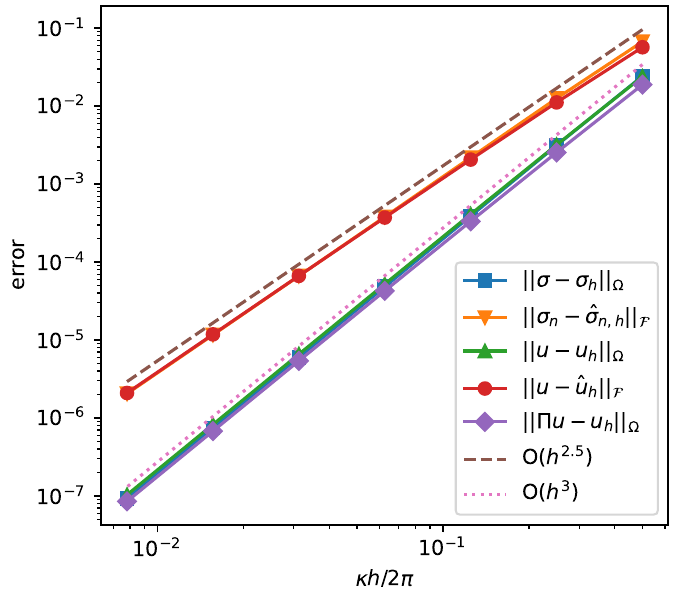}
       		\caption{$p=2$}
    	\end{subfigure}
	\begin{subfigure}[b]{0.49\textwidth}
        		\centering
		\includegraphics[width=\textwidth]{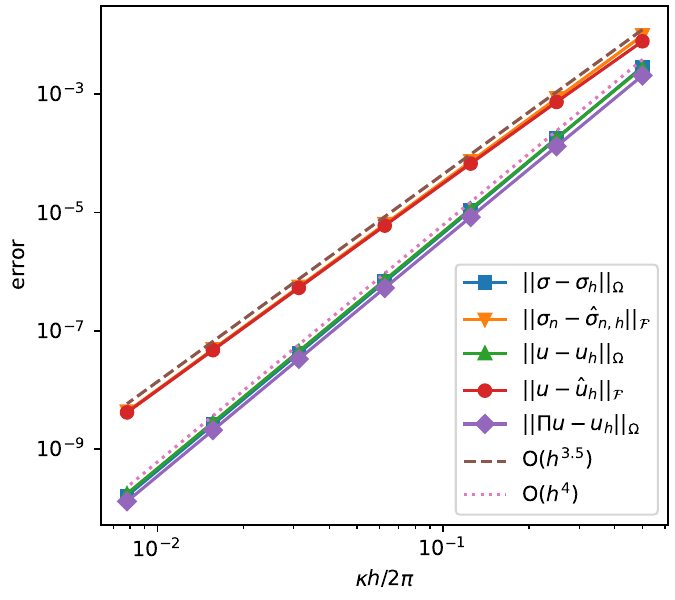}
       		\caption{$p=3$}
    	\end{subfigure}
	\caption{The convergence of the 2D plane wave example on a sequence of structured meshes for various polynomial degrees is visualised. 
		A convergence rate of $h^{p+1}$ can be seen for pressure, flux and the projected error $\Pi u - u_h$.
		The facet variables $\hat \sigma_{\n,h}$ and $\hat u_h$ converge with a rate of $h^{p+1/2}$.}
	\label{fig:2dconv}
\end{figure}
In alignment with the established theory in Section \ref{sec::errorestimatu} the projected error $e_u = \Pi u - u_h$ converges at a rate of $h^{p+1}$.

\subsection{3D Plane Wave}

Similarly to the previous 2D example, a 3D plane wave simulation has been considered on the cubic domain $\domain = (0,1)^3$ discretised with a structured mesh comprised of tetrahedra.
For fixed $N \in \N_{+}$ the domain has been split into $N \times N \times N$ identical cubes and each smaller cube was split into six tetrahedra.
In Figure \ref{fig:3dstructuredmesh} the mesh for $N=2$ is shown.
The excitation $g$ in the HDG formulation \eqref{eq::weakform} has been chosen so that the analytical solution is the plane wave
\begin{align}
	u(x,y,z) = e^{j \kappa \left (x \cos \theta + (y \cos \eta + z \sin \eta) \sin\theta \right)}
\end{align}
with the angles $\theta = \pi / 6$ and $\eta = \pi / 5$. The wave number was fixed at $\kappa = 3$ and a sequence of structured meshes corresponding to $N = \{2, 6, 10, 14, 18, 22, 26\}$ has been considered.
In Figure \ref{fig:2dconv} the convergence rates for the polynomial degrees $p = \{0,1\}$ with respect to the mesh size are visualised.
\begin{figure}
	\centering
	\begin{subfigure}[b]{0.49\textwidth}
        		\centering
		\includegraphics[width=\textwidth]{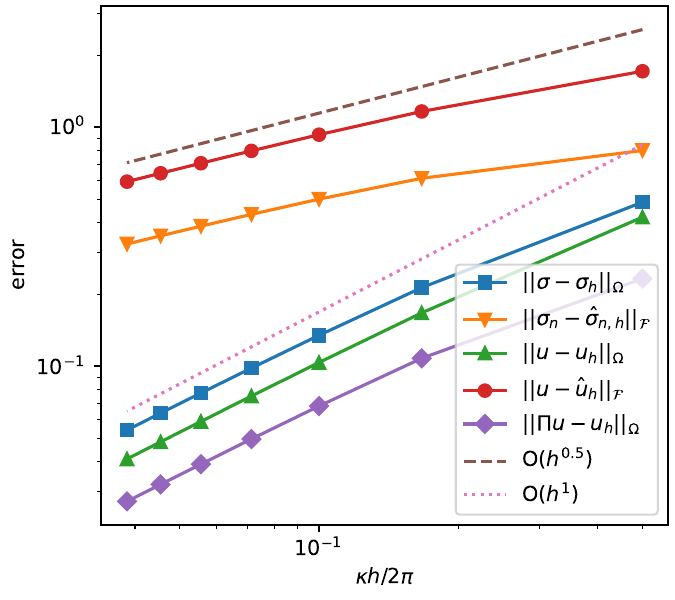}
       		\caption{$p=0$}
    	\end{subfigure}
	\begin{subfigure}[b]{0.49\textwidth}
        		\centering
		\includegraphics[width=\textwidth]{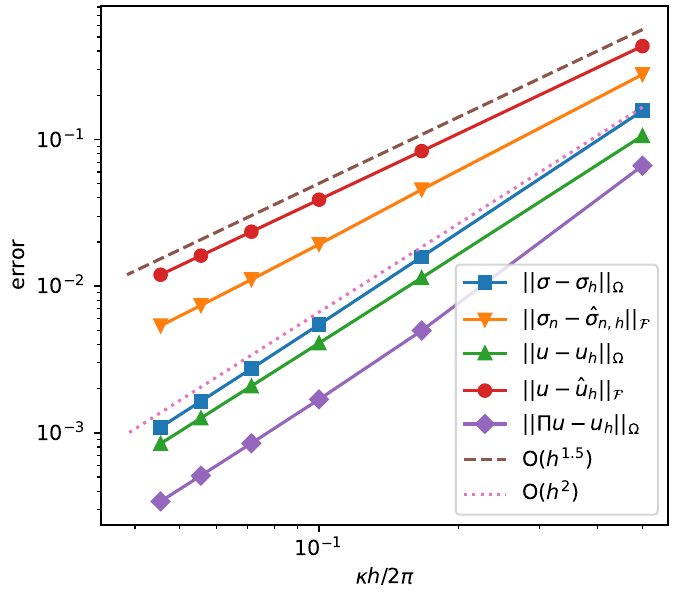}
       		\caption{$p=1$}
    	\end{subfigure}
	\caption{The convergence of the 3D plane wave example on a sequence of structured meshes for polynomial degrees zero and one is visualised. 
		A convergence rate of $h^{p+1}$ can be seen for pressure, flux and the projected error $\Pi u - u_h$.
		The facet variables $\hat \sigma_{\n,h}$ and $\hat u_h$ converge with a rate of $h^{p+1/2}$.}
	\label{fig:2dconv}
\end{figure}

For the following larger simulations, static condensation was applied eliminating the volume variables $\sigma_h$ and $u_h$.
A system of linear equations for the facet variables $\hat \sigma_{\n,h}$ and $\hat u_h$ remained to be solved and the solution was afterwards extended back to the interior of elements.
As an iterative solver a BiCGSTAB with a block Gauss-Seidel preconditioner and a tolerance of $\varepsilon = 10^{-5}$ was applied. In 2D simulations, the facet variables in a vertex patch were combined into blocks
and for 3D simulations, facet-wise blocks were generated.

\subsection{Heterogeneous Materials}

The analysis in this work only covers the case of constant material parameters, but the method can also cope with heterogeneous materials.
In this numerical experiment, the following heterogeneous Helmholtz problem has been considered.

\begin{definition}[Mixed Heterogeneous Helmholtz Problem]
For given $\w > 0$ and $g \in \ltwo(\bnd)$, let $(\sigma, u)$ be the solution of the boundary value problem (BVP)
\begin{subequations}
\begin{align}
	j \w \sigma - \grad u = 0 && \text{ in } \domain, \\
	- \diver \sigma + j \w c u = 0 && \text{ in } \domain, \\
	\sigma \cdot \n + u = g && \text{ on } \bnd,
\end{align}
where $c(\x) > 0$ is a given positive, bounded, varying material coefficient.
\end{subequations}
\end{definition}

The following 2D-examples have similar geometry and material coefficients as \cite[Experiment 6.5]{euanheterogeneous}.
The domain consists of the square $\domain = (-1,1)^2$ with a circle of radius $r = 1 / 2$ as a penetrable obstacle centred in the middle.
The wavenumber has been chosen as $\w = 100$ and the excitation was
\begin{align}
	g(\x) = -10j \w e^{-20 \left ( y + \frac{1}{10} \right )^2}
\end{align}
on the left boundary. On the other outer boundaries, homogenous Robin boundary conditions were applied.
The excitation is slightly offset from the axis of symmetry and represents an inflowing Gaussian peak.
For the discrete FE spaces, a polynomial degree of $p=4$ has been used and the maximal mesh size was chosen as $h= {2 \pi} / {8 \w}$.
Two different material profiles were considered for the simulations, specifically
\begin{align}
	c_1(\x) :=
	\begin{cases}
		2 \sqrt{x^2 + y^2}c_{min} + \left (1 - 2 \sqrt{x^2 + y^2} \right) c_{max}, &\lVert \x \rVert_2 < \frac{1}{2}, \\
		1, &\text{otherwise},
	\end{cases}
\end{align}
and 
\begin{align}
	c_2(\x) :=
	\begin{cases}
		\left (1 - 2 \sqrt{x^2 + y^2} \right) c_{min} + 2 \sqrt{x^2 + y^2}c_{max}, &\lVert \x \rVert_2 < \frac{1}{2}, \\
		1, &\text{otherwise},
	\end{cases}
\end{align}
with the minimum $c_{min} = 0.02$ and the maximum $c_{max} = 50$.
The first material $c_1$ is constant outside the circle and decreases linearly with respect to the radius.
Contrary, the second material $c_2$ increases linearly.
In Figure \ref{fig:matc1} and \ref{fig:matc2} the real part of the pressure can be seen for the simulations with material coefficients $c_1$ and $c_2$.
\begin{figure}
	\centering
	\begin{subfigure}[b]{0.4\textwidth}
        		\centering
		\includegraphics[width=\textwidth]{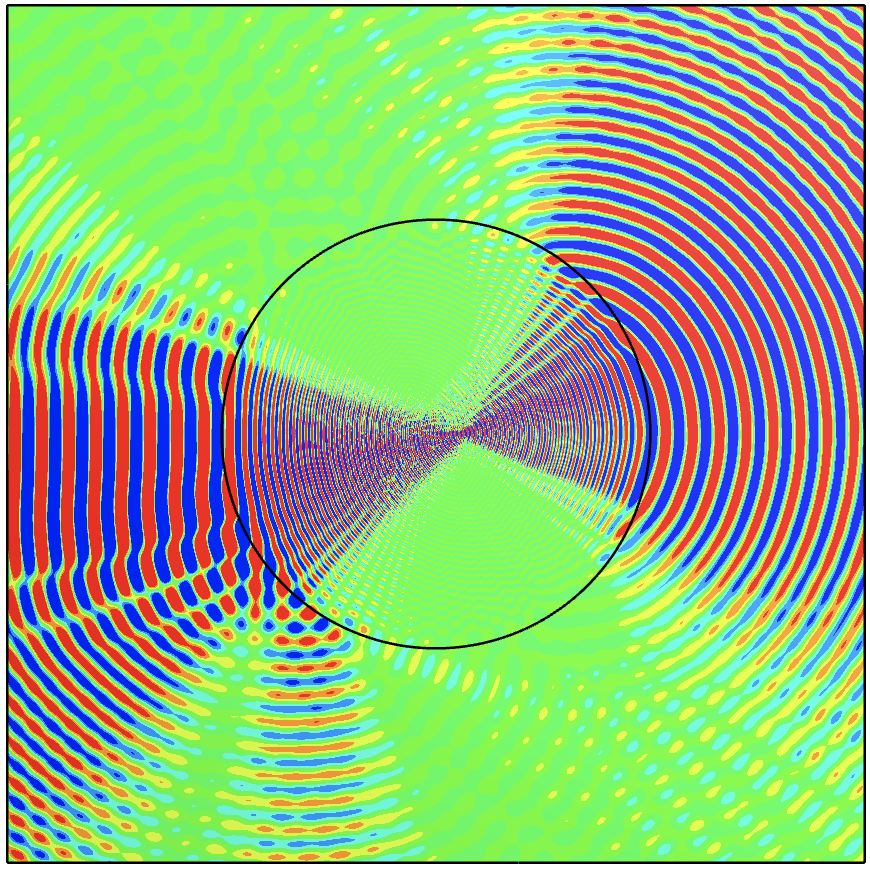}
       		\caption{Result for linearly decreasing material $c_1$ with respect to the radius inside the circle}
      		\label{fig:matc1}
    	\end{subfigure}
	\hspace{2cm}
	\begin{subfigure}[b]{0.4\textwidth}
        		\centering
		\includegraphics[width=\textwidth]{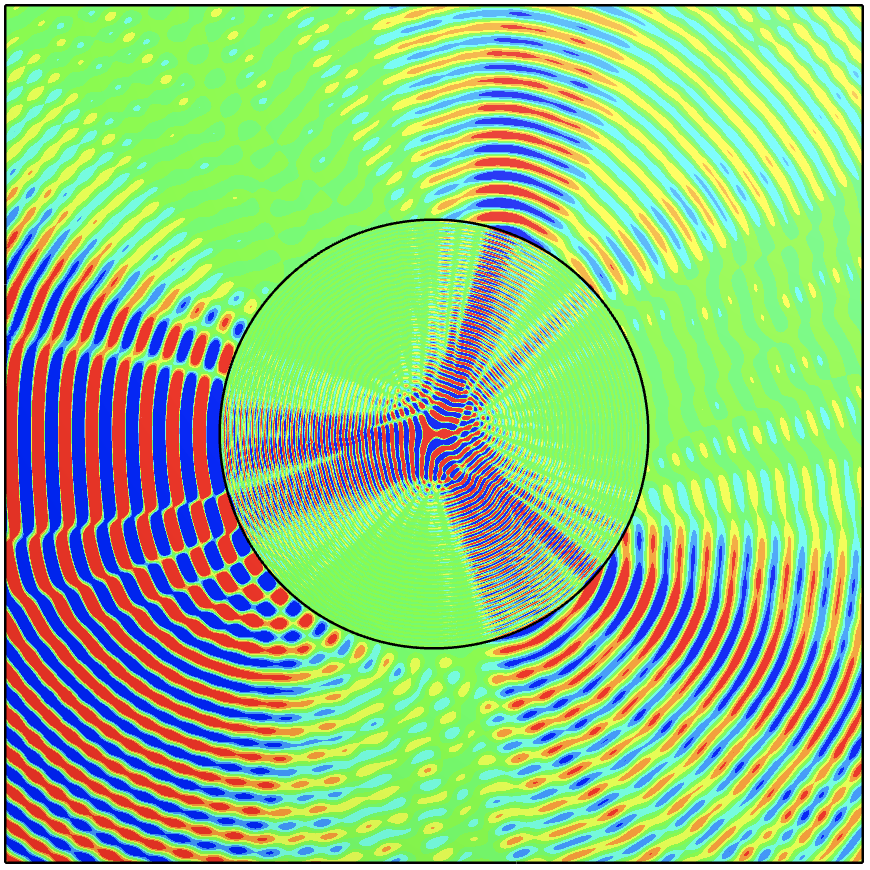}
       		\caption{Result for linearly increasing material $c_2$ with respect to the radius inside the circle}
      		\label{fig:matc2}
    	\end{subfigure}
	\caption{The real part $\Re ( u_h )$ off the pressure is shown for simulations with the heterogeneous material coefficients $c_1$ and $c_2$. A Gaussian peak has been applied on the left boundary as an excitation.}
\end{figure}
Both simulations have been carried out on $8$ cores, required approximately $14 \ \rm{GB}$ of memory 
and the number of degrees of freedom (NDoF) can be found in Table \ref{table:ndofs}. 
After static condensation a system of linear equations with the combined size of $\hat \sigma_{\n,h}$ and $\hat u_h$ was solved.
In Table \ref{table:comptimes} the iteration counts and the computation times for both simulations are shown.
The wall time reflects the duration of the iterative solve and the processor time is the sum of the computation times of all cores combined.

\begin{table}
\centering
\caption{NDoF for the 2D simulation with heterogeneous materials}
\begin{tabular}{|c| c c c c|}
	\hline
	& $\sigma_h$ & $\hat \sigma_{\n,h}$ & $u_h$ & $\hat u_h$ \\
	\hline
	NDoFs & 4 279 920 & 1 072 530 & 2 139 960 & 1 072 530 \\
	\hline
\end{tabular}
\label{table:ndofs}
\end{table}

\begin{table}
\centering
\caption{Number of iterations and computation times for heterogeneous materials}
\begin{tabular}{|c| c c c c|}
	\hline
	& iterations & wall time in \si{\s} & processor time in \si{\s} & processor time per core in \si{\s} \\
	\hline
	$c_1$ & 5 833 & 2 363 & 15 712 & 1 964 \\
	$c_2$ & 4 024 & 1 634 & 10 867 & 1 359 \\
	\hline
\end{tabular}
\label{table:comptimes}
\end{table}

\subsection{Scattering on Spheres}

A 3D example with $100$ spheres as scatterers treated as homogeneous Dirichlet boundaries has been considered. The domain is comprised of the cube $\domain = (0,1)^3$ 
and on the plane of symmetry perpendicular to the $x$-axis an array of the scatterers is positioned.
They are aligned on a $10 \times 10$ grid with an equidistant spacing of $\delta = 1/11$ in between midpoints. The radius $r = \delta/3$ of each sphere is identical.
A cut view of the geometry and the solution can be seen in Figure \ref{fig:3ddomain}.
\begin{figure}
	\centering
	\begin{subfigure}[b]{0.4\textwidth}
        		\centering
		\includegraphics[width=\textwidth]{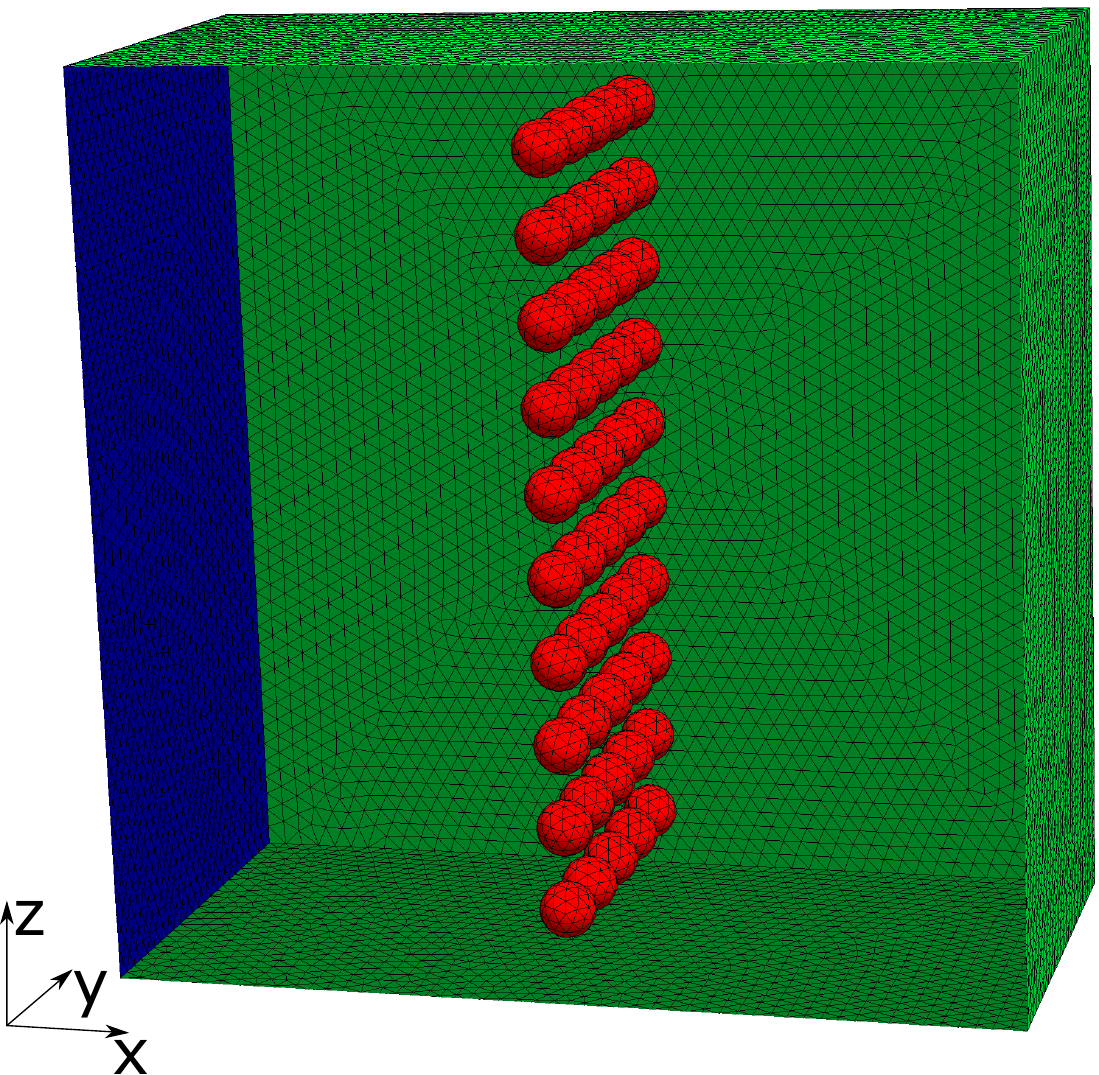}
       		\caption{The spherical scatterers can be seen in red and the left excitation boundary is in blue.}
    	\end{subfigure}
	\hspace{2cm}
	\begin{subfigure}[b]{0.4\textwidth}
        		\centering
		\includegraphics[width=\textwidth]{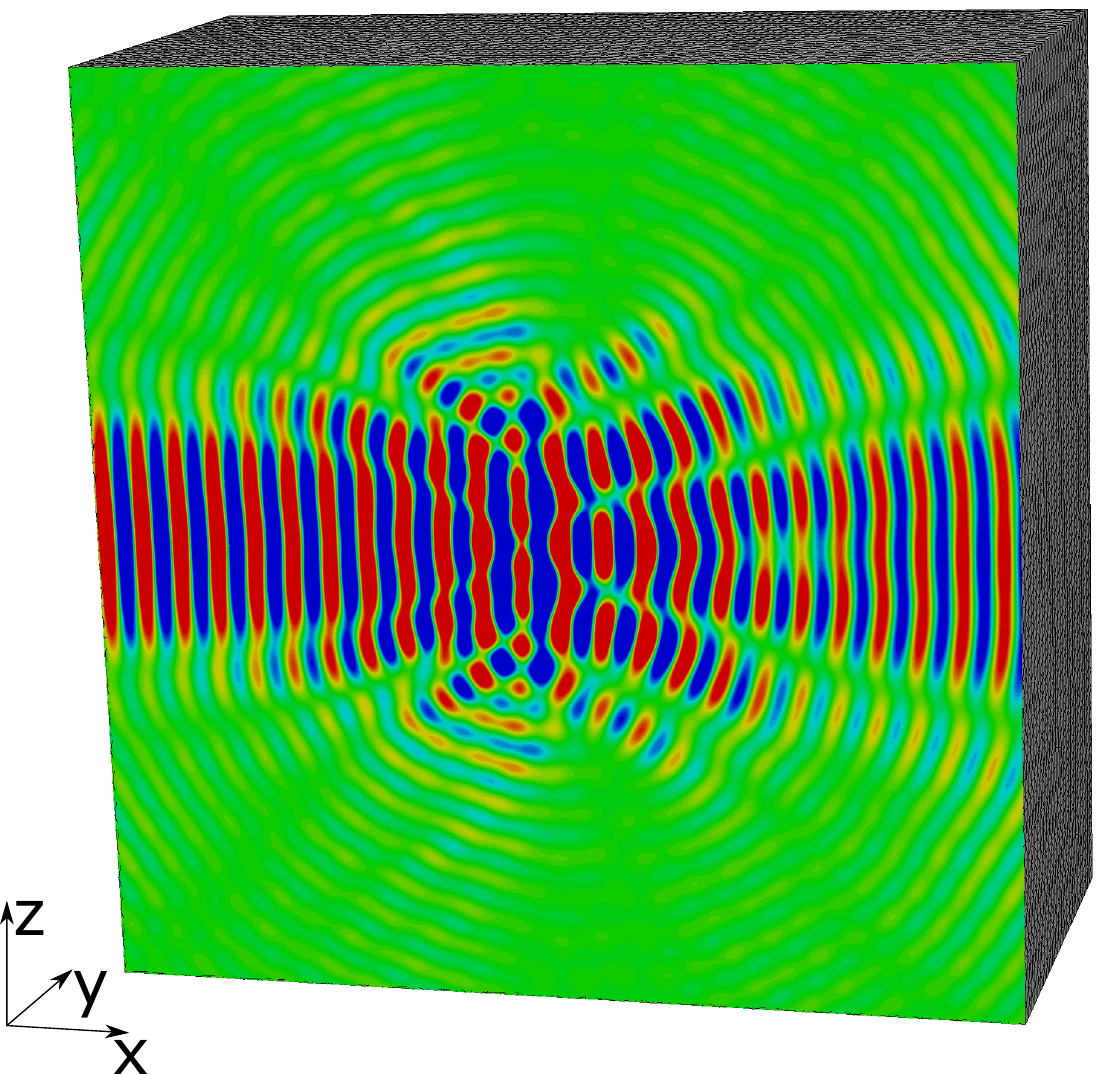}
       		\caption{The real part $\Re ( u_h )$ of the pressure is drawn on the plane of symmetry.}
    	\end{subfigure}
	\caption{A cut view along the $x,z$-plane of symmetry is shown for the 3D-example with spherical scatterers.}
      	\label{fig:3ddomain}
\end{figure}
A wavenumber of $\w = 150$ was considered and the material coefficient was constant $c=1$.
The excitation was the Gaussian peak
\begin{align}
	g(\x) = -10j \w e^{-100 \left ( \left ( y - \frac{1}{2} \right )^2 + \left ( z - \frac{1}{2} \right )^2 \right )}
\end{align}
on the left boundary.
For the discrete space, a polynomial degree of $p=4$ has been used and the maximal mesh size was chosen as $h = 2 \pi /2 \w$.
The simulation has been carried out on $16$ cores, required approximately $153 \ \rm{GB}$ of memory 
%42GB, 300 its, 1280s
and the NDoF can be seen in Table \ref{table:ndofs3d}.
%\begin{table}
%\centering
%\caption{NDoFs for the 3D scattering experiment of flux, pressure and facet variables}
%\begin{tabular}{|c| c c c c|}
%	\hline
%	& $\sigma_h$ & $\hat \sigma_{\n,h}$ & $u_h$ & $\hat u_h$ \\
%	\hline
%	NDoFs & 15 846 600 & 4 685 190 & 5 282 200 & 4 685 190 \\
%	\hline
%\end{tabular}
%\label{table:ndofs3d}
%\end{table}
\begin{table}
\centering
\caption{NDoF for the 3D scattering experiment}
\begin{tabular}{|c| c c c c|}
	\hline
	& $\sigma_h$ & $\hat \sigma_{\n,h}$ & $u_h$ & $\hat u_h$ \\
	\hline
	NDoFs & 56 994 735 & 16 616 295 & 18 998 245 & 16 616 295 \\
	\hline
\end{tabular}
\label{table:ndofs3d}
\end{table}
The solver stopped after 417 iterations with a wall time of $8204$ \si{\s}.
%ptime = 73785

%Discrete spaces
%$u_h \in \mathcal{P}^p (T), \sigma_h \in \mathcal{RT}^p(T), \hat u_h \in \mathcal{P}^p (F), \hat \sigma_{\n,h} \in \mathcal{P}^p (F)$ lead to 
%%Relative errors:
%%\begin{align}
%%	\frac{\lVert u - u_h \rVert_T}{\lVert u \rVert_T} = \Theta (h^{p+1}), &&
%%	\frac{\lVert \sigma - \sigma_h \rVert_T}{\lVert \sigma \rVert_T} = \Theta (h^{p+1}), &&
%%	\frac{\lVert \diver \sigma - \diver \sigma_h \rVert_T}{\lVert \diver \sigma \rVert_T} = \Theta (h^p), \\
%%	\frac{\lVert u - \hat u_h \rVert_\elbnd}{\lVert u \rVert_\elbnd} = \Theta (h^{p + \frac{3}{2}}), &&
%%	\frac{\lVert \sigma \cdot \n - \hat \sigma_{\n,h} \rVert_\elbnd}{\lVert \sigma \cdot \n \rVert_\elbnd} = \Theta (h^{p + \frac{3}{2}}).
%%\end{align}
%absolute errors
%\begin{align}
%	\lVert u - u_h \rVert_T = \Theta (h^{p+1}), &&
%	\lVert \sigma - \sigma_h \rVert_T = \Theta (h^{p+1}), &&
%	\lVert \diver \sigma - \diver \sigma_h \rVert_T = \Theta (h^p), \\
%	\lVert u - \hat u_h \rVert_\elbnd = \Theta (h^{p + \frac{1}{2}}), &&
%	\lVert \sigma \cdot \n - \hat \sigma_{\n,h} \rVert_\elbnd = \Theta (h^{p + \frac{1}{2}})
%\end{align}
%and replacing the $\mathcal{RT}^p (T)$-space with the $\mathcal{BDM}^p(T)$-space leads to
%\begin{align}
%	\lVert u - u_h \rVert_T = \Theta (h^{p+1}), &&
%	\lVert \sigma - \sigma_h \rVert_T = \Theta (h^{p+1}), &&
%	\lVert \diver \sigma - \diver \sigma_h \rVert_T = \Theta (h^p), \\
%	\lVert u - \hat u_h \rVert_\elbnd = \Theta (h^{p + \frac{1}{2}}), &&
%	\lVert \sigma \cdot \n - \hat \sigma_{\n,h} \rVert_\elbnd = \Theta (h^{p + \frac{1}{2}}).
%\end{align}

\section*{Acknowledgment}

The authors are grateful to Jens Markus Melenk for the informative discussions on the topic of the paper,
especially, for pointing out the wave number dependency of the adjoint Helmholtz problem.
This work was supported by the "University SAL Labs" initiative of Silicon Austria Labs (SAL) and its Austrian partner universities for applied fundamental research for electronic-based systems
and the authors acknowledge the support of Bernhard Auinger of SAL.

\bibliographystyle{alpha}
\bibliography{library.bib}

\end{document}